\documentclass[11pt, a4paper]{article} 

\usepackage{amssymb,amsmath,amsthm,color}
\usepackage{latexsym,dsfont,amsfonts,amsmath,amssymb}
\usepackage{mathrsfs}
\usepackage{pdfsync}

\usepackage{mathtools}

\newcommand{\pd}[3]{\frac{\partial ^{#1} #2}{\partial #3}}
\newcommand{\dd}[3]{\dfrac{\mathrm{d}^{#1}#2}{\mathrm{d}#3^{#1}}}

\newcommand{\norm}[2]{\ensuremath{\|#1\|_{#2}}}

\newcommand{\N}[0]{\mathbb{N}}

\newcommand{\R}[0]{\mathbb{R}}

\makeatletter
\newcommand{\BIG}{\bBigg@{3}}
\newcommand{\vast}{\bBigg@{4}}
\newcommand{\Vast}{\bBigg@{5}}
\makeatother

\newcommand{\gammatilde}{\gamma}

\usepackage[normalem]{ulem}

\newtheorem{theorem}{Theorem}
\newtheorem{lemma}[theorem]{Lemma}

\theoremstyle{definition}
\newtheorem{remark}[theorem]{Remark}

\newcommand{\bsDelta}{{\boldsymbol{\Delta}}}

\newcommand{\bsrho}{{\boldsymbol{\rho}}}

\newcommand{\bspsi}{{\boldsymbol{\psi}}}

\newcommand{\bst}{{\boldsymbol{t}}}

\newcommand{\bsx}{{\boldsymbol{x}}}

\newcommand{\bsy}{{\boldsymbol{y}}}

\newcommand{\bsz}{{\boldsymbol{z}}}
\newcommand{\bsA}{{\boldsymbol{A}}}

\newcommand{\bsPhi}{{\boldsymbol{\Phi}}}

\newcommand{\rd}{\mathrm{d}}

\newcommand{\bbR}{\mathbb{R}}

\newcommand{\bbE}{\mathbb{E}}
\newcommand{\bbone}{\mathds{1}}

\newcommand{\calC}{\mathcal{C}}
\newcommand{\setD}{\mathcal{D}}

\newcommand{\calH}{\mathcal{H}}

\newcommand{\calL}{\mathcal{L}}

\newcommand{\calO}{\mathcal{O}}

\newcommand{\calV}{\mathcal{V}}
\newcommand{\calW}{\mathcal{W}}

\newcommand{\tr}{{\sf T}}

\newcommand{\bigO}{\mathcal{O}}

\renewcommand{\forall}{\text{ for all }}

\def\R{\mathbb{R}}

\newcommand{\setu}{{\mathrm{\mathfrak{u}}}}
\newcommand{\setv}{{\mathrm{\mathfrak{v}}}}
\newcommand{\setw}{{\mathrm{\mathfrak{w}}}}

\newcommand{\mask}[1]{{}}

\definecolor{darkred}{RGB}{139,0,0}
\definecolor{darkgreen}{RGB}{0,100,0}
\definecolor{darkmagenta}{RGB}{170,0,120}
\definecolor{darkpurple}{RGB}{110,0,180}
\definecolor{darkblue}{RGB}{40,0,200}
\definecolor{darkbrown}{rgb}{0.75,0.40,0.15}

\newcommand{\be}{\begin{equation}}
\newcommand{\ee}{\end{equation}}
\newcommand{\bea}{\begin{eqnarray}}
\newcommand{\eea}{\end{eqnarray}}
\newcommand{\beas}{\begin{eqnarray*}}
\newcommand{\eeas}{\end{eqnarray*}}

\newlength{\CapLen}
\AtBeginDocument{\settoheight{\CapLen}{A}}

\def\r2p{{\sqrt{2\pi}}}
\newcommand{\subsecintro}[1]{\subsubsection*{#1}}  

\title{Equivalence between Sobolev spaces of first-order dominating mixed smoothness
and unanchored ANOVA spaces on $\mathbb{R}^d$}

\makeatletter
\let\@fnsymbol\@arabic
\makeatother

\author{Alexander D. Gilbert\footnotemark[1] \and
             Frances Y. Kuo\footnotemark[1] \and
             Ian H. Sloan\footnotemark[1]}

\date{\today}

\begin{document}

\maketitle

\footnotetext[1]{School of Mathematics and Statistics, University of New South Wales,
                          Sydney, NSW 2052, Australia\\
                          \texttt{alexander.gilbert@unsw.edu.au}, \texttt{f.kuo@unsw.edu.au},
                           \texttt{i.sloan@unsw.edu.au}
                          }

\begin{abstract}
We prove that a variant of the classical Sobolev space of first-order
dominating mixed smoothness is equivalent (under a certain condition) to
the unanchored ANOVA space on $\mathbb{R}^d$, for $d \geq 1$. Both spaces
are Hilbert spaces involving \emph{weight functions}, which determine the
behaviour as different variables tend to $\pm \infty$, and \emph{weight
parameters}, which represent the influence of different subsets of
variables. The unanchored ANOVA space on $\mathbb{R}^d$ was initially introduced
by Nichols \& Kuo in 2014 to analyse the error of \emph{quasi-Monte Carlo}
(QMC) approximations for integrals on unbounded domains; whereas the
classical Sobolev space of dominating mixed smoothness was used as
the setting in a series of papers by Griebel, Kuo \& Sloan on the
smoothing effect of integration, in an effort to develop a rigorous theory
on why QMC methods work so well for certain non-smooth integrands with
\emph{kinks} or \emph{jumps} coming from option pricing problems. In this
same setting, Griewank, Kuo, Le\"ovey \& Sloan in 2018 subsequently
extended these ideas by developing a practical \emph{smoothing by
preintegration} technique to approximate integrals of such functions with
kinks or jumps.

We first prove the equivalence in one dimension (itself a non-trivial
task), before following a similar, but more complicated, strategy to prove
the equivalence for general dimensions.
As a consequence of this equivalence, we analyse applying QMC combined
with a preintegration step to approximate the fair price of an Asian
option, and prove that the error of such an approximation using $N$ points
converges at a rate close to $1/N$.
\end{abstract}

\section{Introduction}

In this paper we establish equivalence between a variant of the classical
Sobolev (Hilbert) space $\calH = \calH_d$ of real-valued functions with
first-order dominating mixed smoothness on $\R^d$, and a reproducing
kernel Hilbert space $\calW = \calW_d$ introduced in \cite{NK14}.
Throughout this paper we will refer to $\calH$ as a ``\emph{Sobolev
space}'', and refer to $\calW$ as an ``\emph{ANOVA space}'' due to an
intimate connection with the \emph{ANOVA decomposition} of functions. 
(Specifically, the ANOVA decomposition of a function in $\calW$, or $\calH$,
is an orthogonal decomposition with respect to the inner product in $\calW$,
see \eqref{eq:norm-decomp} below). 
Both spaces involve \emph{weight functions} (see $\rho_j$ and
$\psi_j$ below) to control the behaviour of the functions and their mixed
derivatives as the $j$th variable $x_j$ goes to $\pm\infty$. Both spaces
also involve \emph{weight parameters} (see $\gamma_\setu$ below) that
moderate the relative importance of subsets of variables. We prove that
the spaces $\calH$ and $\calW$ are equivalent provided that the weight
functions satisfy a certain condition (see \eqref{eq:strong} below).

\subsecintro{Motivation}

The motivation for this work requires a side trip in this introduction.
Since the mid 1990s, \emph{Quasi-Monte Carlo} (QMC) methods
\cite{DKS13,DP10,Nie92,SJ94} have been a powerful tool for practitioners,
but presented a challenge to theorists, arising from the unexpected
success of QMC methods in tackling practical high dimensional integrals
from mathematical finance---a popular example being the pricing of
path-dependent options. Many research papers, e.g.,
\cite{AcBroadGlass97,BoyBroadGlass97,LemLEc98,PapTraub96,PasTraub95,WangSloan06},
demonstrated empirically that applying QMC methods to a range of finance
problems gave significantly faster convergence rates than the commonly
used \emph{Monte Carlo} (MC) simulations. Yet it is surprising that QMC
beats MC for these problems, since both the classical and more recent QMC
theory cannot be applied, because the integrands typically involve
``kinks'' or ``jumps'' (i.e., integrands and/or their first partial
derivatives are not continuous) and so fail to satisfy the smoothness
requirements of the theory.

One approach to explaining the success of QMC for finance problems was by
way of the concept of \emph{effective dimension}
\cite{CMO97,WangFang03}. In principle any $d$-variate function can be
expressed in terms of its unique ANOVA (\emph{ANalysis Of VAriance})
decomposition
\[
  f = \sum_{\setu\subseteq\setD} f_\setu,
\]
where $\setD \coloneqq \{1,2,\ldots,d\}$, each term $f_\setu$ depends only on the
variables $x_j$ with indices $j\in\setu$, and the $2^d$ terms are
$\calL^2$ orthogonal (with respect to a weight function; see below). It is
generally accepted that QMC works well when $f$ has a low \emph{truncation
dimension} (i.e., $f$ is dominated by the contributions from ANOVA terms
involving only a small number of early variables) or a low
\emph{superposition dimension} (i.e., $f$ is dominated by a few ANOVA
terms, each involving only a small number of variables). Although in those
papers an explicit ANOVA decomposition was not carried out in practice,
the handwaving justification for the QMC success was that, behind the
scenes, QMC (for some initially unexplained reason) approximates well the
low-dimensional contributions to the integral, while the remaining
high-dimensional contributions collectively make a negligible
contribution. Recent work has developed rigorous definitions of truncation
\cite{KritzPillWas16} and superposition \cite{GilWas17} dimensions in
certain function space settings,
but for finance applications the problem remains that typical integrands
fail to satisfy the smoothness assumptions.

A series of papers \cite{GKS10,GKS13,GKS17note,GKS17} provided partial
justification for the ``low effective dimension'' argument, by proving
that most of the ANOVA terms of the option pricing integrands are smooth.
Specifically, it was proved in \cite{GKS17note} that, with the single
exception of the very last term with $\setu = \setD$, all $2^d-1$ other
ANOVA terms $f_\setu$ belong to the Sobolev space $\calH$ (details to be
given below). A subsequent paper \cite{GKLS18} took the theory one step
further by developing a practical \emph{smoothing by preintegration}
technique for cubature over $\R^d$, whereby a strategically chosen
coordinate is integrated out first (either analytically or numerically
using a high precision $1$-dimensional quadrature rule) to yield a new
function involving the remaining $d-1$ variables, and a high-dimensional
cubature rule can then be applied to the resulting ``preintegrated''
function, which by the theory in \cite{GKLS18} was shown to belong to the
Sobolev space $\calH_{d-1}$. However, the paper \cite{GKLS18} could not at
that time be used to guarantee the success of QMC combined with
preintegration because the necessary QMC analysis had been carried out not
in $\calH$ but in $\calW$.

Since QMC theory has the unit cube as its natural setting, its extension
to the unbounded region $\R^d$ has necessitated development of a new
theoretical setting. In particular, the paper \cite{NK14} derived a new
reproducing kernel Hilbert space (RKHS) $\calW$ and proved that the
generating vector for a \emph{randomly shifted lattice rule} can be
constructed using a component-by-component algorithm to achieve the
optimal rate of convergence.

In summary, on the one hand much progress has been made in justifying the
``low effective dimension" argument \cite{GKS10,GKS13,GKS17note,GKS17} and
then in developing from it a practical preintegration technique
\cite{GKLS18}, all in the setting of the Sobolev space~$\calH$; while on
the other hand a constructive QMC theory has been developed in the setting
of~$\calW$~\cite{NK14}. But this QMC methodology could not validly be used
with preintegration without knowledge of the relationship between the two
spaces.

This issue is completely resolved in the present paper. We show that the
two spaces $\calH$ and $\calW$ are indeed equivalent, under an appropriate
condition on the weight functions, with embedding constants expressed
explicitly in terms of the weight functions and weight parameters. As a
result there is now available a complete QMC-based strategy, with solid
theoretical foundations, for tackling the high-dimensional integrals
arising from option pricing.

Before moving on, we note that the basic idea of \emph{smoothing by
preintegration} is not original to \cite{GKLS18}, and is a special case of
\emph{conditioning} or \emph{conditional sampling}, see, e.g.,
\cite[Sec.~7.2.3]{Glasserman}. Indeed, several conditional sampling
methods using different quadrature rules (such as MC, QMC and sparse
grids) have previously been applied to option pricing problems in, e.g.,
\cite{ACN13a,ACN13b,BST17,GlaSta01,GKLS18,Hol11,WWH17}. The contribution
of \cite{GKLS18} was to formalise the notion of preintegration, and prove
that the $(d-1)$-dimensional preintegrated function will be sufficiently
smooth.

Having provided the background motivation for this work, in the remainder
of this introduction we summarise our key results, and discuss how our
present paper relates to recent work
\cite{GnewHefHinRit17,GnewHefHinRitWas17,GnewHefHinRitWas19,
HefRit15,HefRitWas16,HinSchneid16,KritzPillWas17} on embeddings of similar
spaces.

\subsecintro{The $1$-dimensional case}

Let $\rho: \R \to \R_+$ be a probability density function defined on $\R$,
and let $\psi: \bbR\to \R_+$ be a locally integrable function such
that $1/\psi$ is also locally integrable. Let $\gamma>0$ be a weight
parameter (it plays little role in one dimension). Starting from the set
of all locally integrable functions on $\bbR$, the Sobolev space $\calH$
and the unanchored ANOVA space $\calW$ each contains those functions for
which the respective norm is finite:
\begin{align} \label{eq:Hnorm1D}
 \|f\|_\calH^2
 &\,\coloneqq\, \int_{-\infty}^\infty |f(x)|^2\, \rho(x)\,\rd x\;\;
 + \frac{1}{\gamma} \int_{-\infty}^\infty | f'(x)|^2\,\psi(x)\,\rd x,
 \\
 \label{eq:Wnorm1D}
 \|f\|_\calW^2
 &\,\coloneqq\, \,
 \bigg|\int_{-\infty}^\infty f(x)\, \rho(x)\,\rd x\bigg|^2 \,
 + \frac{1}{\gamma} \int_{-\infty}^\infty |f'(x)|^2\, \psi(x)\,\rd x.
\end{align}
Here $f'$ is the weak derivative, which is defined to be the
locally integrable function satisfying
\[
\int_{-\infty}^\infty f'(x) v(x) \, \rd x 
\,=\, 
-\int_{-\infty}^\infty f(x) \dd{}{}{x} v(x) \, \rd x
\]
for all smooth functions $v$ with compact support.
Thus the functions in $\calH$ are square-integrable against the weight
function $\rho$, and their first derivatives are square-integrable against
the weight function $\psi$. On the other hand, the functions in $\calW$
only needs to be integrable against the weight function $\rho$. We
summarise this as
\begin{align*}
 f\in \calH &\iff f\in \calL^2_\rho \quad\mbox{and}\quad f'\in\calL^2_\psi, \\
 f\in \calW &\iff f\in \calL^1_\rho \quad\mbox{and}\quad f'\in\calL^2_\psi.
\end{align*}
It follows from the Cauchy--Schwarz inequality that $\|f\|_\calW^2 \le
\|f\|_\calH^2$ and so $\calH$ is embedded in $\calW$. Obviously $\calH$ is
embedded in $\calL^2_\rho$, but $\calW$ may or may not be embedded in
$\calL^2_\rho$. Since trivially $\|f\|_\calW^2 \le \|f\|_\calH^2 \le
\|f\|_{\calL^2_\rho}^2 + \|f\|_\calW^2$, we see that the two spaces
$\calH$ and $\calW$ are equivalent if and only if $\calW$ is embedded in
$\calL^2_\rho$.

We remark that the two norms \eqref{eq:Hnorm1D} and \eqref{eq:Wnorm1D}
differ in just their first terms, but to establish the norm equivalence we
need to make use of their common second term involving the derivative
of~$f$. This hinges upon the interplay between the two weight functions
$\rho$ and $\psi$.

Let
\begin{align} \label{eq:cdf}
  \Phi(x) \,\coloneqq\, \int_{-\infty}^x \rho(t) \,\rd t
\end{align}
denote the distribution function corresponding to the density $\rho$.
Throughout we consider two different conditions on the relationship
between the weight functions $\rho$ and $\psi$. In Section~\ref{sec:1D} we
show that $\calW$ is a RKHS if the pair of weight functions $(\rho,\psi)$
satisfies the \textbf{weaker condition}
\begin{equation} \label{eq:weak}
 \int_{-\infty}^{c} \frac{(\Phi(t))^2}{\psi(t)} \,\rd t < \infty
 \quad\mbox{and}\quad
 \int_{c}^\infty \frac{(1 - \Phi(t))^2}{\psi(t)} \,\rd t < \infty
 \quad\mbox{for all finite $c$}.
\end{equation}
Furthermore, $\calW$ is embedded in $\calL^2_\rho$ if the pair
$(\rho,\psi)$ satisfies the \textbf{stronger condition}
\begin{equation} \label{eq:strong}
 \int_{-\infty}^{c} \frac{\Phi(t)}{\psi(t)} \,\rd t < \infty
 \quad\mbox{and}\quad
 \int_{c}^\infty \frac{1 - \Phi(t)}{\psi(t)} \,\rd t < \infty
 \quad\mbox{for all finite $c$}.
\end{equation}
This allows us in this paper under the stronger condition to establish the
\textbf{norm equivalence}
\begin{align} \label{eq:equiv1D}
  \|f\|_\calW^2 \,\le\, \|f\|_\calH^2 \,\le\, (1 + \gamma\,C(\rho,\psi))\, \|f\|_\calW^2,
\end{align}
with
\begin{equation} \label{eq:C-def}
  C(\rho,\psi) \,\coloneqq\, \int_{-\infty}^\infty \frac{\Phi(t)\,(1 - \Phi(t))}{\psi(t)} \,\rd t \,<\, \infty,
\end{equation}
where finiteness is due to \eqref{eq:strong}. Hence we conclude that
\emph{the two spaces $\calH$ and $\calW$ are equivalent under the stronger
condition \eqref{eq:strong}}.

The reproducing kernel for $\calW$ was derived in \cite{NK14} under the
assumption that the stronger condition \eqref{eq:strong} holds from the
outset, and so $\calW$ is embedded in $\calL^2_\rho$ by assumption. The
question of whether the reproducing property exists under the weaker
condition \eqref{eq:weak} was not considered in that paper. Moreover,
the results in \cite{NK14} were not proved in the generality claimed there
and this is repaired in the current work.
Note additionally that in this paper we write $\psi(x)$ instead of
$(\psi(x))^2$, which was the notation in \cite{NK14}.

The condition \eqref{eq:weak} on its own is not sufficient to establish the equivalence of $\calW$ and
 $\calH$.  This can be seen by choosing $\rho$ to be the standard normal density and 
 $\psi(x)\coloneqq \exp(-\tfrac{3}{4} x^2)$, noting that with this choice \eqref{eq:weak} is satisfied but 
 \eqref{eq:strong} is not; yet the function $f(x)= 1/\sqrt{\rho(x)}$ belongs to $\calW$, but 
not to $\calH$, making $\calW$ strictly larger than $\calH$.  
This example also shows that the equivalence shown in this paper 
is not a trivial consequence of standard embeddings. 

\subsecintro{The $d$-dimensional case}

Consider now a potentially different pair of weight functions
$(\rho_j,\psi_j)$ for each coordinate index $j=1,\ldots,d$, and a weight
parameter $\gamma_\setu$ for every subset $\setu\subseteq\setD$. As in
\eqref{eq:cdf} we denote the distribution function of $\rho_j$ by $\Phi_j$
for each $j$.

In \cite{NK14} the ANOVA space $\calW$ was extended to $d$ dimensions by
defining its reproducing kernel to be a particular sum of products of
$1$-dimensional kernels. This particular representation gives an
impression that the resulting function space $\calW_d$ may not include all
functions in the classical Sobolev space $\calH_d$. To close this
loophole, in this paper we will not define $\calW_d$ in terms of the
reproducing kernel, and additionally do not assume any product structure.
Instead we will define the spaces in complete analogy to our one
dimensional case, as follows.

Starting from the set of all locally integrable functions on $\bbR^d$, the
Sobolev space $\calH_d$ and the unanchored ANOVA space $\calW_d$ each
contains those functions for which the respective norm is finite:
\begin{align}
\label{eq:H-norm}
 \|f\|_{\calH_d}^2
 &\coloneqq \sum_{\setu\subseteq\setD} \frac{1}{\gamma_\setu}
 \int_{\R^d} |\partial^\setu f(\bsx)|^2\,
 \bspsi_\setu(\bsx_\setu)\, \bsrho_{\setD\setminus\setu}(\bsx_{\setD\setminus\setu}) \,\rd \bsx,
 \\
 \|f\|_{\calW_d}^2 \label{eq:W-norm}
 &\coloneqq
 \sum_{\setu\subseteq\setD} \frac{1}{\gamma_\setu}
 \int_{\R^{|\setu|}}
 \!\bigg| \int_{\R^{d-|\setu|}} \!\partial^\setu f(\bsx_\setu,\bsx_{\setD\setminus\setu})\,
 \bsrho_{\setD\setminus\setu}(\bsx_{\setD\setminus\setu})\,\rd\bsx_{\setD\setminus\setu}\bigg|^2
 \bspsi_\setu(\bsx_\setu) \,\rd\bsx_{\setu},
\end{align}
where $\partial^\setu f$ denotes the \emph{weak derivative} (see
\eqref{eq:weak_D} below) of $f$ respect to the ``active variables'' (the
ones being differentiated), $\bsx_\setu \coloneqq \{x_j : j\in\setu\}$,
which in turn are weighted by the product $\bspsi_\setu(\bsx_\setu) \coloneqq
\prod_{j\in\setu} \psi_j(x_j)$, while the ``inactive variables'' are
weighted by the product
$\bsrho_{\setD\setminus\setu}(\bsx_{\setD\setminus\setu}) \coloneqq
\prod_{j\in\setD\setminus\setu} \rho_j(x_j)$. Since each $\rho_j$ is a
probability density function, the Cauchy--Schwarz inequality implies that
$\|f\|_{\calW_d}^2 \le \|f\|_{\calH_d}^2$ and therefore $\calH_d$ is
embedded in $\calW_d$. We also know that $\calH_d$ is embedded in
$\calL^2_\bsrho$, with $\bsrho(\bsx) \coloneqq \prod_{j=1}^d \rho_j(x_j)$. The
question is again whether or not $\calW_d$ is embedded in
$\calL^2_{\bsrho}$.

We prove in Section~\ref{sec:HD} that $\calW_d$ is a reproducing kernel
Hilbert space if the weaker condition \eqref{eq:weak} holds for all pairs
of weight functions $(\rho_j,\psi_j)$, and furthermore that $\calW_d$ is
indeed embedded in $\calL^2_{\bsrho}$ if the stronger condition
\eqref{eq:strong} holds for all pairs $(\rho_j,\psi_j)$. In turn, with the
condition \eqref{eq:strong} we prove the \textbf{norm equivalence}
\begin{align} \label{eq:equiv}
  \|f\|_{\calW_d}^2 \,\le\, \|f\|_{\calH_d}^2 \,\le\,
  \bigg(\max_{\setv\subseteq\setD} \sum_{\setw\subseteq\setv} \frac{\gamma_\setv}{\gamma_{\setv\setminus\setw}}
   \prod_{j\in\setw} C(\rho_j,\psi_j) \bigg)\, \|f\|_{\calW_d}^2,
\end{align}
where $C(\rho_j,\psi_j)$ is defined as in \eqref{eq:C-def} for each $j$.
In the special case of \emph{product weights}, i.e., there is a weight
parameter $\gamma_j$ associated with each coordinate $x_j$, and
$\gamma_\setu \coloneqq \prod_{j\in\setu} \gamma_j$, the embedding constant
(squared) in \eqref{eq:equiv} is precisely
\[
  \prod_{j=1}^d (1 + \gamma_j\, C(\rho_j,\psi_j)),
\]
which is simply the product of the constant in \eqref{eq:equiv1D}, and is
bounded independently of $d$ provided that $\sum_{j=1}^\infty \gamma_j\,
C(\rho_j,\psi_j)<\infty$.

Knowing an explicit and simple formula for the reproducing kernel of
$\calW_d$ (see Theorem~\ref{thm:main} below) allowed the development of
QMC theory in \cite{NK14}, namely, the construction of randomly shifted
lattice rules that achieve the optimal rate of convergence. Note that if
\eqref{eq:strong} holds then, because $\calH_d$ with inner product
corresponding to the norm \eqref{eq:H-norm} is equivalent to $\calW_d$, we
conclude that $\calH_d$ is also a reproducing kernel Hilbert space, but
with a kernel that is unknown as well as likely more complicated, hence
our preference for working with $\calW_d$.

\subsecintro{Implication for smoothing by preintegration applied to option pricing problems}

As discussed earlier in this introduction, there is a gap in the analysis
of QMC methods combined with preintegration. The theory on smoothing by
preintegration from \cite{GKLS18} exists for the space $\calH_d$, whereas
the error analysis of QMC methods giving a root-mean-square (RMS) error
close to $\bigO(1/N)$, where $N$ is the number of function evaluations,
assumes that the integrand belongs to the space $\calW_d$ (see
\cite[Theorem~8]{NK14}). The equivalence of the two spaces $\calH_d$ and
$\calW_d$ bridges this gap, and an important consequence is we can now
show that QMC methods combined with preintegration can achieve a RMS error
close to $\bigO(1/N)$ for some option pricing problems; explicit details
are given in Section~\ref{sec:option}.

\subsecintro{Other embedding and related results}

In a series of related papers
\cite{GnewHefHinRit17,GnewHefHinRitWas17,GnewHefHinRitWas19,
HefRit15,HefRitWas16,HinSchneid16,KritzPillWas17} different combinations
of authors established, in a variety of settings, continuous embeddings between ANOVA
spaces and so-called ``anchored'' spaces.  While all of these papers
considered functions defined on $d$-dimensional spaces, it is easier to
explain the concept in the case $d=1$.  In this case the anchored
equivalent (with anchor at zero) of the squared norms defined in
\eqref{eq:Hnorm1D} and \eqref{eq:Wnorm1D} is
\begin{equation*}
 \|f\|_{\rm{anch}}^2
 \,\coloneqq\, |f(0)|^2
 + \frac{1}{\gamma} \int_{-\infty}^\infty |f'(x)|^2\,\psi(x)\,\rd x.
\end{equation*}
In the present paper we do not consider anchored spaces.
Note also that by continuous embedding 
we mean that the identity mapping from one space into the other
is a bounded linear operator. Often we will simply use the term
\emph{embedding}, which should be understood as a continuous embedding.

The first such paper \cite{HefRit15} studies embeddings of tensor products
of $1$-dimensional Hilbert spaces with product weights $\{\gamma_j\}$. The
setting in one dimension is quite general, but explicit examples cover
only the bounded domain $[0,1]$. It is possible to put our
$1$-dimensional spaces $\calH$ and $\calW$ into the setting of
\cite{HefRit15}, however due to the tensor product structure used there it is not
possible to extend those results to our $d$-dimensional spaces. 
Furthermore, the constants arising from the general theory in \cite{HefRit15}
are not as sharp as those we obtain in \eqref{eq:equiv1D}, see also 
Remark~\ref{rem:embed-constant}.
The paper \cite{HefRitWas16} extends \cite{HefRit15} to Banach spaces
involving $\calL^p$ norms on $[0,1]^d$ with general weights
$\{\gamma_\setu\}$, and provides embedding constants between the ANOVA and
anchored spaces for $p=1$ and $p=\infty$. The paper \cite{HinSchneid16}
extends \cite{HefRitWas16} to general $p\in [1,\infty]$ by interpolation.
The paper \cite{KritzPillWas17} provides lower bounds on the norm of the
embedding operator, again for bounded domains and ANOVA and anchored
spaces. The paper \cite{GnewHefHinRitWas17} extends \cite{HefRit15} to
higher order derivatives and $d=\infty$. The paper
\cite{GnewHefHinRitWas19} considers spaces with increasing smoothness and
$d=\infty$.

Apart from \cite{HefRit15}, the other paper close to the present work is 
\cite{GnewHefHinRitWas17}, in that it deals with general weights and unbounded domain, though with
integrals restricted to $\mathbb{R}_+^d$. The function spaces are also
defined in a different way, through convolutions with integral kernels
rather than derivatives, and $\psi$ is restricted to be equal (in our
notation) to $\rho$. The condition on $\rho$ assumed in that paper can be
stated as
\[
 \bigg\|\frac{1-\Phi}{\rho}\bigg\|_{\calL^\infty(\mathbb{R}_+)}<\infty \qquad\mbox{and}\qquad
 \int_0^\infty x\,\rho(x) \,\rd x <\infty.
\]
It can easily be seen that this condition (when adapted to the whole real
line) is stronger than our condition \eqref{eq:weak} but weaker than
\eqref{eq:strong}. Clearly, embedding theory for Sobolev-type spaces in
high dimensions is an active area of research, however, it would seem not
possible to infer the results in the present paper from the equivalence
results for ANOVA and anchored spaces.

Another possible explanation for the success of QMC for option pricing
was proposed in \cite{HMOU16}, where it was suggested that \emph{Besov} spaces
are more suitable for the analysis of functions with kinks.
However, \cite{HMOU16} deals only with Besov spaces of periodic functions on
the unit cube and only considers products of simple kink functions on $[0, 1]$.
As such, the analysis there does not apply to real-world option pricing problems.

\section{The $1$-dimensional case} \label{sec:1D}

Let $\rho: \R \to \R_+$ be a strictly positive probability density
function defined on $\R$, and let $\psi: \bbR\to \R_+$ be a locally
integrable strictly positive function such that $1/\psi$ is also locally
integrable. For any locally integrable function $f$ on $\bbR$, we define
the $\rho$-weighted integral
\[
  I_\rho(f) \,\coloneqq\, \int_{-\infty}^\infty f(x)\, \rho(x) \,\rd x,
\]
along with the $\calL^2_\rho$ and $\calL^2_\psi$ norms
\[
 \|f\|_{\calL^2_{\rho}}^2 \,\coloneqq\, \int_{-\infty}^\infty |f(x)|^2\,\rho(x) \,\rd x,
\qquad \|f\|_{\calL^2_{\psi}}^2 \coloneqq \int_{-\infty}^\infty |f(x)|^2\,\psi(x) \,\rd x.
\]
The Sobolev space $\calH$ and the unanchored ANOVA space $\calW$ are the
restriction of the set of locally integrable functions on $\bbR$ for which
the norms \eqref{eq:Hnorm1D} and \eqref{eq:Wnorm1D},
respectively, are finite. Equivalently, the norms can be written as
\begin{align}
 \|f\|_\calH^2 &\,=\, \;\|f\|_{\calL^2_\rho}^2\,\, + \frac{1}{\gamma}\, \|f'\|_{\calL^2_\psi}^2, \label{eq:Hnorm1D2} \\
 \|f\|_\calW^2 &\,=\, |I_\rho(f)|^2 + \frac{1}{\gamma}\, \|f'\|_{\calL^2_\psi}^2. \label{eq:Wnorm1D2}
\end{align}
Thus for $f\in\calH$ we have $f\in \calL^2_\rho$ and $f'\in\calL^2_\psi$,
and the norm in $\calH$ is a weighted $\calL^2$-norm involving first
derivatives, but with differentiated functions weighted differently from
undifferentiated functions. On the other hand, for $f\in\calW$ we have
$f\in \calL^1_\rho$ and $f'\in\calL^2_\psi$. 
It is well-known that $\calH$ is complete, see, e.g., \cite[Section~1.1.12]{Mazya11}.
Since we could not find a proof that $\calW$ is complete in the literature, we provide
one in the following lemma.

\begin{lemma}
\label{lem:W_complete}
If the condition \eqref{eq:weak} holds, then the space $\calW$ is complete.
\end{lemma}

\begin{proof}
Let $\{f_k\}_{k = 1}^\infty$ be a Cauchy sequence in $\calW$. Since, for any $k, \ell \in \N$,
\[
|I_\rho(f_k) - I_\rho(f_\ell)| \,\leq\, \|f_k - f_\ell\|_\calW
\quad \text{and} \quad
\|f_k' - f_\ell'\|_{\calL_\psi^2} \,\leq\, \sqrt{\gamma}\,\|f_k - f_\ell\|_\calW\,,
\]
it follows that $\{ I_\rho(f_k)\}$ is a Cauchy sequence in $\R$
and $\{f_k'\}$ is a Cauchy sequence in $\calL_\psi^2(\R)$.
We denote the respective limits by 
\begin{equation}
\label{eq:lims_W_complete}
\lim_{k \to \infty} I_\rho(f_k) \,\eqqcolon\, I \in \R
\quad \text{and} \quad
\lim_{k \to \infty} f_k' \,\eqqcolon\,  g \in \calL_\psi^2(\R).
\end{equation}

We prove that $f_k$ converges to some $f \in \calW$. Let $a \in \R$ and define $f : \R \to \R$ by
\[
f(x) \,\coloneqq\, \int_a^x g(t) \, \rd t + B\,,
\quad \text{with} \quad
B \,\coloneqq\, I - \int_{-\infty}^\infty \bigg(\int_a^x g(t) \, \rd t\bigg) \rho(x) \, \rd x\,,
\]
implying $I_\rho(f) = I$ because $B$ is finite, as we will now show.
To that end, consider
\begin{align*}
&\int_{-\infty}^\infty \bigg(\int_a^x g(t) \, \rd t\bigg) \rho(x) \, \rd x
\\
&= \int_{-\infty}^a \bigg(- \int_x^a g(t) \, \rd t \bigg) \rho(x) \, \rd x
+ \int_a^\infty \bigg(\int_a^x g(t) \, \rd t\bigg) \rho(x) \, \rd x,
\end{align*}
where we have split the outer integral based on whether $x \leq a$ or $x > a$
and swapped the limits for the inner integral in the first term.
Taking the absolute value then using the triangle inequality and Fubini's Theorem
we obtain
\begin{align*}
&\bigg|\int_{-\infty}^\infty \bigg(\int_a^x g(t)\, \rd t\bigg) \rho(x) \, \rd x\bigg|
\\
&\leq\, 
\bigg| \int_{-\infty}^a g(t) \bigg(\int_{-\infty}^t \rho(x) \, \rd x \bigg) \rd t \bigg|
+ \bigg| \int_a^\infty g(t) \bigg(\int_t^\infty \rho(x) \, \rd x\bigg) \rd t\bigg| 
\\
&=\, \bigg| \int_{-\infty}^a g(t) \Phi(t) \, \rd t\bigg| 
+ \bigg| \int_a^\infty g(t) \big(1 - \Phi(t)\big) \, \rd t \bigg| 
\\
&\leq\, \bigg( \int_{-\infty}^a \frac{[\Phi(t)]^2}{\psi(t)} \, \rd t\bigg)^{1/2}
\bigg( \int_{-\infty}^a |g(t)|^2 \psi(t) \, \rd t\bigg)^{1/2}
\\
&\qquad+ \bigg(\int_a^\infty \frac{[1 - \Phi(t)]^2}{\psi(t)} \, \rd t\bigg)^{1/2}
\bigg( \int_a^\infty |g(t)|^2 \psi(t) \, \rd t\bigg)^{1/2}
\\
&\leq\, \|g\|_{\calL_\psi^2} \Bigg[
\bigg( \int_{-\infty}^a \frac{[\Phi(t)]^2}{\psi(t)} \, \rd t\bigg)^{1/2}
+ \bigg(\int_a^\infty \frac{[1 - \Phi(t)]^2}{\psi(t)} \, \rd t\bigg)^{1/2}\Bigg]
\,<\, \infty,
\end{align*}
the finiteness following because $g \in \calL_\psi^2(\R)$ and by assumption \eqref{eq:weak}.
Hence, $|B|<\infty$.

Next we show $f \in \calW$.
Since $g \in \calL_\psi^2(\R)$, it follows that $g$ is also locally integrable.
Indeed, let $\Omega \subset \R$ be compact then because
$1/\psi$ is locally integrable
\begin{align*}
\int_\Omega g(t) \, \rd t 
\,&\leq\,
\bigg(\int_\Omega |g(t)|^2 \psi(t) \, \rd t\bigg)^{1/2}
\bigg( \int_\Omega \frac{1}{\psi(t)} \, \rd t\bigg)^{1/2}
\\
\,&\leq\, \|g\|_{\calL_\psi^2} \bigg( \int_\Omega \frac{1}{\psi(t)} \, \rd t\bigg)^{1/2}
\,<\, \infty.
\end{align*}

Since $g$ is locally integrable, it follows that $f$ is absolutely continuous,
and by the Fundamental Theorem of Calculus the classical derivative of $f$
is equal to $g$ almost everywhere.
Since the classical derivative exists, the weak derivative exists as well 
and is equal to the classical derivative almost everywhere.
This implies that $f' = g$ almost everywhere in $\R$.
Then using \cite[Sec.~1.1.2, Theorem]{Mazya11}, $f'$ locally integrable
implies that $f$ is also locally integrable.

Hence, $f \in \calW$ since
\[
\|f\|_\calW^2 \,=\, |I_\rho(f)|^2 + \|f'\|_{\calL_\psi^2}^2
\,=\, |I|^2 + \|g\|_{\calL_\psi^2}^2 \,<\, \infty.
\]
Finally, by construction $f_k \to f$ in $\calW$ because \eqref{eq:lims_W_complete}
implies
\[
\|f - f_k \|_\calW^2 \,=\, |I_\rho(f - f_k)|^2 + \|f' - f_k'\|_{\calL_\psi^2}^2
\,=\, |I - I_\rho(f_k)|^2 + \|g - f_k'\|_{\calL_\psi^2}^2
\,\to\, 0,
\]
as $k \to \infty$. Thus, every Cauchy sequence in $\calW$ converges 
and so $\calW$ is complete.
\end{proof}

It follows from the definition of the norms that the functions in $\calH$
and $\calW$ are continuous and absolutely continuous. The spaces $\calH$
and $\calW$ are both Hilbert spaces, with the respective inner products
\begin{align}
\langle f, \widetilde{f} \rangle_{\calH} \,\coloneqq\,
 &\int_{-\infty}^\infty f(x)\, \widetilde{f}(x)\,\rho(x)\,\rd x
 + \frac{1}{\gamma}\int_{-\infty}^\infty f'(x)\, \widetilde{f}\,'(x)\,\psi(x) \,\rd x,
 \label{eq:inner1DH}
 \\
  \label{eq:inner1DW}
 \langle f, \widetilde{f} \rangle_{\calW} \,\coloneqq\,
 &\bigg(\int_{-\infty}^\infty f(x)\,\rho(x)\,\rd x \bigg)\bigg(\int_{-\infty}^\infty \widetilde{f}(x)\,\rho(x)\,\rd x \bigg)
 \\\nonumber
 &+ \frac{1}{\gamma}\int_{-\infty}^\infty f'(x)\,\widetilde{f}\,'(x)\,\psi(x) \,\rd x.
\end{align}

The question to be addressed is whether the spaces are identical; a
question whose answer is not obvious even in this $1$-dimensional case. In
one direction, it is immediately clear that $\calH$ is a subset of
$\calW$, in that by the Cauchy--Schwarz inequality we have
\begin{align*}
  \big|I_\rho(f)\big|^2
  &= \bigg| \int_{-\infty}^\infty \! f(x)\,\rho(x)\,\rd x \bigg|^2
  \le \bigg(\int_{-\infty}^\infty \! |f(x)|^2\, \rho(x)\,\rd x\bigg)
  \bigg(\int_{-\infty}^\infty \!\rho(x)\,\rd x\bigg)
  = \|f\|_{\calL^2_\rho}^2.
\end{align*}
We conclude easily that $\|f\|_\calW^2 \le \|f\|_\calH^2$, and therefore
the space $\calH$ is embedded in~$\calW$. The central purpose of this
section is to prove that embedding holds in the opposite direction under
an appropriate condition, see \eqref{eq:strong}.

The parameter $\gamma>0$ in \eqref{eq:Hnorm1D}--\eqref{eq:Wnorm1D} and
\eqref{eq:Hnorm1D2}--\eqref{eq:inner1DW} is a weight parameter that
controls the contribution of $\|f'\|_{\calL^2_\psi}^2$ relative to
$\|f\|_{\calL^2_\rho}^2$ or $(I_\rho(f))^2$ for functions in the unit ball
of $\calH$ or $\calW$. A small $\gamma$ means that
$\|f'\|_{\calL^2_\psi}^2$ must be small. In the limiting case where
$\gamma = 0$, we assume that $f'\equiv 0$ so the spaces contain constant
functions. The weight parameter $\gamma$ does not play an essential role
in the $1$-dimensional setting but it will become significant later when
we extend the setting to higher dimensions.

Let $\Phi$ denote the distribution function corresponding to the density
$\rho$, see \eqref{eq:cdf}. In Subsection~\ref{sec:kernel1D} we will show
that $\calW$ is a RKHS if the pair of weight functions $(\rho,\psi)$
satisfies the weaker condition \eqref{eq:weak}. In
Subsection~\ref{sec:embed1D}, we will show that $\calW$ is embedded in
$\calL^2_\rho$ if the pair $(\rho,\psi)$ satisfies the stronger condition
\eqref{eq:strong}. In turn we will establish the norm equivalence between
$\calH$ and~$\calW$.

\subsection{The reproducing kernel for $\calW$ exists under the weaker condition \eqref{eq:weak}} \label{sec:kernel1D}

Before we proceed to find the reproducing kernel for $\calW$, we give a
couple of remarks on the conditions \eqref{eq:weak}, \eqref{eq:strong}
and~\eqref{eq:C-def}.

\begin{remark}
Since $0 \leq\Phi(t) \leq 1$, the condition \eqref{eq:strong} is clearly
stronger than the condition \eqref{eq:weak}. We now verify that
\eqref{eq:strong} and \eqref{eq:C-def} are equivalent. For any finite $c$
\begin{align*}
 \int_{-\infty}^\infty \frac{\Phi(t)\,(1 - \Phi(t))}{\psi(t)} \,\rd t
 &\,=\, \int_{-\infty}^c \frac{\Phi(t)\,(1 - \Phi(t))}{\psi(t)} \,\rd t
 + \int_c^\infty \frac{\Phi(t)\,(1 - \Phi(t))}{\psi(t)} \,\rd t \\
 &\,\le\, \int_{-\infty}^c \frac{\Phi(t)}{\psi(t)} \,\rd t
 + \int_c^\infty \frac{1 - \Phi(t)}{\psi(t)} \,\rd t,
\end{align*}
which shows that \eqref{eq:strong} implies \eqref{eq:C-def}. Since $\Phi$
is the cumulative distribution function of the probability density $\rho$,
there exists some $a \leq c$ such that $\Phi(a) < 1$. Using
\begin{align*}
  \Phi(x) \,=\, \frac{\Phi(x)\,(1-\Phi(x))}{1-\Phi(x)}
  \,\le\, \frac{\Phi(x)\,(1-\Phi(x))}{1-\Phi(a)}
  \qquad\mbox{for all}\quad x \in (-\infty,a],
\end{align*}
we obtain
\begin{align*}
\int_{-\infty}^c \frac{\Phi(x)}{\psi(x)} \,\rd x
&\,=\, \int_{-\infty}^a \frac{\Phi(x)}{\psi(x)} \,\rd x + \int_a^c \frac{\Phi(x)}{\psi(x)} \,\rd x\\
&\,\le\, \frac{1}{1 - \Phi(a)} \int_{-\infty}^c \frac{\Phi(x)\,(1 - \Phi(x))}{\psi(x)} \,\rd x
+ \int_a^c \frac{1}{\psi(t)} \, \rd x,
\end{align*}
where the first term is finite if \eqref{eq:C-def} holds and the second is
finite because $1/\psi$ is locally integrable. Finiteness of the second
integral in \eqref{eq:strong} can be shown in a similar way. This shows
that \eqref{eq:C-def} implies \eqref{eq:strong} and hence they are
equivalent.
\end{remark}

\begin{remark}
The condition \eqref{eq:strong}, while not very restrictive, is
nevertheless not trivial. In the important case in which $\rho$ is the
standard normal probability density
$\rho(x)\coloneqq\tfrac{1}{\sqrt{2\pi}}\exp(-\tfrac{1}{2}x^2)$, the choice
$\psi(x) = \rho(x)$ satisfies \eqref{eq:weak} but does not satisfy
\eqref{eq:strong}. This can be seen from the bounds
\[
  \tfrac{1}{t}\big(1-\tfrac{1}{t^2}\big)\,\exp\big(-\tfrac{1}{2}t^2\big)
  \,\le\, \int_t^\infty \exp\big(-\tfrac{1}{2}x^2\big)\,\rd x
  \,\le\, \tfrac{1}{t} \exp\big(-\tfrac{1}{2}t^2\big)
  \quad\forall\, t>0,
\]
which are obtained using asymptotic expansions from, e.g.,
\cite{AbrSteg70}. Indeed, if $\psi(x) = \exp(-\tfrac{1}{2\alpha}x^2)$ then
\eqref{eq:weak} holds if and only if $\alpha \ge 1/2$, but
\eqref{eq:strong} holds if and only if $\alpha> 1$.
\end{remark}

We first establish the following technical lemma. This result was used
implicitly in \cite{NK14} without proof.

\begin{lemma} \label{lem:lim1D}
Let $f:\bbR\to\bbR$ be locally integrable, with $f\in \calL^1_\rho$ and
$f'\in\calL^2_\psi$. If the condition~\eqref{eq:weak} holds, then
\begin{align}
 \label{eq:lim1Da}
 f(x)\, \Phi(x) \,\to\, 0 &\quad\text{as}\quad x \,\to\, -\infty, \qquad\text{and}\\
\label{eq:lim1Db}
 f(x)\, (1 - \Phi(x)) \,\to\, 0 &\quad\text{as}\quad x \,\to\, + \infty.
\end{align}
\end{lemma}

\begin{proof}
We prove the first limit \eqref{eq:lim1Da} by contradiction; the proof for
the second limit \eqref{eq:lim1Db} follows analogously.

To simplify the notation, we define $F(x) \coloneqq f(x)\,\Phi(x)$ so that
we must now show that $F(x) \to 0$ as $x \to -\infty$. Note $F$ is
absolutely continuous because both $f$ and $\Phi$ are absolutely
continuous, and since $\Phi$ is the distribution function of $\rho$, the
derivative of $F$ is given by
\[
 F'(x) \,=\,  f'(x)\,\Phi(x) + f(x)\,\rho(x).
\]

We first show that $F'$ is integrable on $(-\infty, c]$ for any finite
$c\in\bbR$. By the triangle inequality and the Cauchy--Schwarz inequality,
we can bound the integral by
\begin{align*}
 \int_{-\infty}^c &|F'(x)|\,\rd x
 \,\le\, \int_{-\infty}^c |f'(x)|\,\Phi(x) \,\rd x
 \,+\, \int_{-\infty}^c |f(x)|\,\rho(x) \,\rd x \\
 &\,\le\, \bigg(\int_{-\infty}^c |f'(x)|^2\, \psi(x)\,\rd x \bigg)^{1/2}
 \bigg( \int_{-\infty}^c \frac{(\Phi(x))^2}{\psi(x)} \,\rd x\bigg)^{1/2}
 + \int_{-\infty}^c |f(x)|\,\rho(x) \,\rd x \\
 &\,\leq \|f'\|_{\calL^2_\psi}
 \bigg( \int_{-\infty}^c \frac{(\Phi(x))^2}{\psi(x)} \,\rd x\bigg)^{1/2} + \|f\|_{\calL^1_\rho}
 \,<\, \infty,
\end{align*}
where the finiteness follows from the assumptions $f\in \calL^1_\rho$,
$f'\in \calL^2_\psi$ and \eqref{eq:weak}.
Hence, $F'$ is integrable on the
interval $(-\infty, c]$.

For a contradiction, suppose that \eqref{eq:lim1Da} does not hold, in
which case there exists $\delta > 0$, $M \in \N$ and a sequence $\{t_m\}$
with $t_m \to -\infty$ as $m \to \infty$, such that $t_m\le c$ and
$|F(t_m)|= |f(t_m)|\,\Phi(t_m) \geq \delta$ for all $m \geq M$.

For any $m \geq M$, let $x \in (-\infty, t_m)$ be arbitrary. Then by the
Fundamental Theorem of Calculus we can write
\[
F(x) \,=\, F(t_m) - \int_x^{t_m} F'(t) \,\rd t.
\]
By the reverse triangle inequality, we can bound this from below as
follows
\begin{align*}
|F(x)| \,&\geq\, |F(t_m)| - \bigg|\int_x^{t_m} F'(t) \,\rd t\bigg|
\,\geq\, \delta - \int_{x}^{t_m} |F'(t)| \,\rd t
\,\geq\, \delta - \int_{-\infty}^{t_m} |F'(t)| \,\rd t.
\end{align*}
Since $F'$ is integrable on $(-\infty, c]$, we now can choose $ m\ge M$
large enough such that
\[
\int_{-\infty}^{t_m} |F'(t)|\,\rd t \,\leq\, \frac{\delta}{2},
\]
in which case we have the bound
\[
|F(x)| \,\geq\, \frac{\delta}{2}
\quad \text{for all } x \le t_m.
\]
This in turn implies that
\[
|f(x)| \,\geq\, \frac{\delta}{2\Phi(x)}
\quad \text{for all } x \le t_m.
\]
Thus
\begin{align*}
 \|f\|_{\calL^1_\rho}
 &\,\geq\, \int_{-\infty}^{t_m} |f(x)|\,\rho(x)\,\rd x
 \,\geq\, \frac{\delta}{2}  \int_{-\infty}^{t_m} \frac{\rho(x)}{\Phi(x)} \,\rd x
 \,=\, \frac{\delta}{2}  \int_{-\infty}^{t_m} \dd{}{}{x} \big[\log (\Phi(x))\big] \,\rd x.
\end{align*}
The last integral is divergent, as can be seen by
\begin{align} \label{eq:div1D}
\nonumber
\int_{-R}^{t_m} \dd{}{}{x} \big[\log (\Phi(x))\big] \,\rd x \,&=\, \log(\Phi(t_m)) - \log(\Phi(-R))\\
\,&=\,  \log(\Phi(t_m)) + \log\Big(\frac{1}{\Phi(-R)}\Big)
\to \infty \quad \text{as } R \to \infty,
\end{align}
because $\Phi(-R) \to 0$ as $R \to \infty$. This contradicts the fact that
$f \in \calL^1_\rho$ and so the result \eqref{eq:lim1Da} must hold.
\end{proof}

We now arrive at the main result of this subsection, which proves that the
reproducing kernel exists for $\calW$ under the weaker condition
\eqref{eq:weak}. This result seems not previously known because in
\cite{NK14} the stronger condition \eqref{eq:strong} was assumed from the
outset.

\begin{theorem}
If the condition \eqref{eq:weak} holds, then $\calW$ is a reproducing
kernel Hilbert space with kernel
\begin{equation} \label{eq:ker1D}
K(x, y) \,\coloneqq\, 1 + \gamma\, \eta(x, y),
\end{equation}
where
\begin{align} \label{eq:eta1Da}
 \eta(x, y)
 \,\coloneqq\,&\int_{-\infty}^{\min(x, y)} \frac{(\Phi(t))^2}{\psi(t)} \,\rd t
 + \int_{\max(x, y)}^\infty \frac{(1 - \Phi(t))^2}{\psi(t)} \,\rd t
\\\nonumber
 &- \int_{\min(x,y)}^{\max(x,y)} \frac{\Phi(t)\,(1 - \Phi(t))}{\psi(t)} \,\rd t.
\end{align}
\end{theorem}

\begin{proof}
First we observe that $K$ is symmetric, and can be written in two
other ways:
\begin{align}
 \eta(x, y)
 &=\int_{-\infty}^{\max(x, y)} \frac{(\Phi(t))^2}{\psi(t)} \rd t
 + \int_{\max(x, y)}^\infty \!\!\frac{(1 - \Phi(t))^2}{\psi(t)} \rd t
 - \int_{\min(x,y)}^{\max(x,y)} \frac{\Phi(t)}{\psi(t)} \rd t \label{eq:eta1Dmax}
 \\
 &\le\, \eta(\max(x,y),\max(x,y)), \nonumber
 \end{align}
 and
 \begin{align}
 \eta(x, y)
 &=\int_{-\infty}^{\min(x, y)} \!\frac{(\Phi(t))^2}{\psi(t)} \rd t
 + \!\int_{\min(x, y)}^\infty \!\!\!\frac{(1 - \Phi(t))^2}{\psi(t)} \rd t
 - \!\int_{\min(x,y)}^{\max(x,y)} \frac{1-\Phi(t)}{\psi(t)} \rd t \label{eq:eta1Dmin}
 \\
 &\le\, \eta(\min(x,y),\min(x,y)). \nonumber
\end{align}
Moreover, we have
\begin{align*}
 0\,\le\, \eta(y,y)
 &\,=\, \int_{-\infty}^y \frac{(\Phi(t))^2}{\psi(t)} \,\rd t
 + \int_{y}^\infty \frac{(1-\Phi(t))^2}{\psi(t)} \,\rd t \,<\,\infty,
\end{align*}
where the integrals are finite due to \eqref{eq:weak}.
Thus $K(x,y)<\infty$ for all $x,y\in\bbR$.

To verify that $\calW$ is a RKHS with kernel given by \eqref{eq:ker1D}, we
have to show that (i) $K(\cdot, y) \in \calW$ for all $y \in \R$, and (ii)
$f(y) = \langle f, K(\cdot, y)\rangle_{\calW}$ for all $f \in \calW$ and
$y \in \R$.

For any $y\in\bbR$ we use \eqref{eq:eta1Dmax} for $x\le y$ and
\eqref{eq:eta1Dmin} for $x > y$ to write
\begin{align} \label{eq:zero1D}
\int_{-\infty}^\infty &\eta(x, y)\,\rho(x) \,\rd x
 \nonumber\\
 &=\, \int_{-\infty}^y  \bigg(\int_{-\infty}^{y} \frac{(\Phi(t))^2}{\psi(t)} \,\rd t
 + \int_{y}^\infty \frac{(1 - \Phi(t))^2}{\psi(t)} \,\rd t
 - \int_{x}^{y} \frac{\Phi(t)}{\psi(t)} \,\rd t \bigg)\rho(x) \,\rd x \nonumber\\
 &\quad + \int_y^\infty  \bigg(\int_{-\infty}^y \frac{(\Phi(t))^2}{\psi(t)} \,\rd t
 + \int_y^\infty \frac{(1 - \Phi(t))^2}{\psi(t)} \,\rd t
 - \int_y^x \frac{1 - \Phi(t)}{\psi(t)} \,\rd t \bigg)\rho(x) \,\rd x \nonumber\\
 &=\, \int_{-\infty}^y \frac{(\Phi(t))^2}{\psi(t)} \,\rd t
 + \int_{y}^\infty \frac{(1 - \Phi(t))^2}{\psi(t)} \,\rd t \nonumber\\
 &\quad - \int_{-\infty}^y  \frac{\Phi(t)}{\psi(t)} \bigg(\int_{-\infty}^t \rho(x)\,\rd x \bigg) \,\rd t
 - \int_y^\infty  \frac{1-\Phi(t)}{\psi(t)} \bigg(\int_t^\infty \rho(x) \,\rd x \bigg) \,\rd t \nonumber\\
 &=\, 0,
\end{align}
where in the second step we used the Fubini theorem to change the order of
integration. Thus
\begin{equation} \label{eq:intK1D}
 \int_{-\infty}^\infty K(x,y)\, \rho(x) \,\rd x
   \,=\, 1 + \gamma\,\int_{-\infty}^\infty \eta(x, y)\,\rho(x) \,\rd x \,=\, 1.
\end{equation}

Now we have
\begin{align}
\label{eq:DK}
 \pd{}{}{x} K(x, y) \,&=\, \gamma \pd{}{}{x} \eta(x, y),
  \quad\text{with}
\\
\label{eq:Deta}
 \pd{}{}{x} \eta(x, y) \,&=\,
 \bbone_{(-\infty, y)}(x)\, \frac{\Phi(x)}{\psi(x)}
 \,-\, \bbone_{(y, \infty)}(x)\, \frac{1-\Phi(x)}{\psi(x)},
\end{align}
which follows easily by differentiating \eqref{eq:eta1Dmax} for $x\le y$
and \eqref{eq:eta1Dmin} for $x > y$. Thus
\begin{align} \label{eq:intDeta}
 \int_{-\infty}^\infty \bigg|\pd{}{}{x} \eta(x, y)\bigg|^2 \psi(x) \,\rd x
 \,&=\, \int_{-\infty}^y \frac{(\Phi(x))^2}{\psi(x)} \,\rd x
 + \int_{y}^\infty \frac{(1-\Phi(x))^2}{\psi(x)} \,\rd x
 \nonumber\\
&=\, \eta(y,y),
\end{align}
and so
\begin{align} \label{eq:dKL2norm}
  \bigg\|\pd{}{}{x} K(\cdot, y)\bigg\|_{\calL^2_\psi}^2
  \,=\, \int_{-\infty}^\infty \bigg|\pd{}{}{x} K(x, y)\bigg|^2 \psi(x) \,\rd x
  \,=\, \gamma^2\,\eta(y,y).
\end{align}
Hence we conclude that
\begin{equation}\label{eq:K1norm}
 \|K(\cdot,y) \|_{\calW}^2 \,=\, 1 + \gamma\,\eta(y,y) \,<\, \infty,
\end{equation}
and $K(\cdot, y) \in \calW$ for all $y \in \R$ as required for (i).

For the reproducing property (ii), we have from \eqref{eq:intK1D} that for
any $f\in\calW$ and $y\in\bbR$,
\begin{equation} \label{eq:repro1D}
 \langle f, K(\cdot, y) \rangle_{\calW} \,=\,
 \int_{-\infty}^\infty f(x)\,\rho(x) \,\rd x
 + \frac{1}{\gamma}
 \int_{-\infty}^\infty f'(x)\, \bigg( \pd{}{}{x} K(x, y)\bigg)\, \psi(x) \,\rd x.
\end{equation}
Using \eqref{eq:DK} and \eqref{eq:Deta}, then applying integration by
parts, we obtain
\begin{align*}
 &\frac{1}{\gamma} \int_{-\infty}^\infty f'(x)\,\bigg( \pd{}{}{x} K(x, y) \bigg)\, \psi(x) \,\rd x\\
 &=\, \lim_{R\to\infty}
 \bigg(\int_{-R}^y f'(x)\, \Phi(x) \,\rd x \bigg)
 - \lim_{R\to\infty}\bigg(\int_{y}^R f'(x)\, (1 - \Phi(x)) \,\rd x \bigg) \\
 &=\, \lim_{R\to\infty}
 \bigg(f(y)\,\Phi(y) - f(-R)\,\Phi(-R) - \int_{-R}^y f(x)\,\rho(x)\,\rd x \bigg)\\
 &\qquad
- \lim_{R\to\infty}\bigg(f(R)\,(1-\Phi(R)) + f(y)\,(1-\Phi(y)) - \int_y^R f(x)\, \rho(x) \,\rd x \bigg)\\
 &=\, f(y) - \int_{-\infty}^\infty f(x)\, \rho(x) \,\rd x,
\end{align*}
where the boundary terms vanish as $R\to\infty$ by Lemma~\ref{lem:lim1D}.
Substituting this into \eqref{eq:repro1D} gives the reproducing property
(ii).

Point evaluation is now clearly bounded, since for all $f\in\calW$ and
$y\in\bbR$ we have from \eqref{eq:K1norm}
\begin{align*}
  \big(f(y)\big)^2 \,=\, \langle f, K(\cdot, y) \rangle_{\calW}^2
  \,\le\, \|f\|_\calW^2\,\|K(\cdot, y)\|_\calW^2
   \,=\, \|f\|_\calW^2\, \big(1+\gamma\,\eta(y,y)\big) \,<\,\infty.
\end{align*}
This completes the proof.
\end{proof}

Before leaving the reproducing kernel properties it is convenient to
observe that the subspace $\calV\subset \calW$ defined by
\begin{equation}\label{eq:calVdef}
  \calV \,\coloneqq\, \bigg\{f\in \calW: \int_{-\infty}^\infty f(t)\,
  \rho(t)\,\rd t = 0 \bigg\}
\end{equation}
is a RKHS when equipped with the inner product
\[
 \langle f, \widetilde{f} \rangle_\calV
 \,\coloneqq\, \int_{-\infty}^\infty f'(x)\,\widetilde{f}\,'(x)\,\psi(x)\,\rd x,
 \qquad f, \widetilde{f} \,\in\,\calV.
\]
Note first that for $f,\widetilde{f} \in \calV$ we have, by definition of
the $\calW$ inner product,
\[
 \langle f,\widetilde{f}\rangle_\calW \,=\, \frac{1}{\gamma}\langle f,\widetilde{f}\rangle_\calV.
\]
Thus for $f \in \calV \subset \calW $ we have
\begin{equation}\label{eq:reproV}
 f(y)
 \,=\, \langle f, K(\cdot,y) \rangle_\calW
 \,=\, \gamma\, \langle f, \eta(\cdot,y) \rangle_\calW
 \,=\, \langle f, \eta(\cdot,y)\rangle_\calV,
\end{equation}
from which it follows that $\eta(x,y)$ is the reproducing kernel in
$\calV$ (noting that we have already proved in \eqref{eq:zero1D} that
$\eta(\cdot,y)\in \calV$).

In fact, the space $\calW$ is the direct sum of $\calV$ and the space of
constant functions, since an arbitrary $f \in \calW$ can be written
uniquely as
\begin{equation}\label{eq:anova1}
 f \,=\, f_\emptyset + f_{\{1\}},
\end{equation}
where $f_{\{1\}} \in \calV$ and $f_\emptyset$ is a constant function given
by
\[
f_\emptyset \,=\, \int_{-\infty}^\infty f(x)\, \rho(x) \,\rd x.
\]
(We shall see in Section~\ref{sec:HD} that \eqref{eq:anova1} is a special
case of the ANOVA decomposition.  We are anticipating here the notation to
be used in Section~\ref{sec:HD}.)

\subsection{Norm equivalence in $\calH$ and $\calW$ under the stronger condition
\eqref{eq:strong}} \label{sec:embed1D}

In this Subsection we assume the stronger condition \eqref{eq:strong}.
Under this condition it is easily seen that the kernel $\eta$ defined in
\eqref{eq:eta1Da} can be rewritten as
\begin{align} \label{eq:eta1Db}
 \eta(x, y)
 &\,=\, \int_{-\infty}^{\min(x, y)} \frac{\Phi(t)}{\psi(t)} \,\rd t
 + \int_{\max(x, y)}^\infty \frac{1 - \Phi(t)}{\psi(t)} \,\rd t
 - C(\rho, \psi),
\end{align}
where $C(\rho,\psi)$ is defined in \eqref{eq:C-def}.

We will now show that the norms in $\calH$ and $\calW$ are equivalent. The
embedding constant in \eqref{eq:embedding1} will be improved in
Theorem~\ref{thm:equiv1D}.

\begin{lemma} \label{lem:L2emb}
If the condition \eqref{eq:strong} holds then
\begin{align} \label{eq:embedding1}
  \|f\|_{\calL^2_\rho}^{2}
  &\,\le\, \big(1 + \gamma\,C(\rho,\psi)\big)\,\|f\|_\calW^2
 \qquad\mbox{for all}\quad f \in \calW, \quad\mbox{and}\quad \nonumber\\
  \|f\|_{\calH}^2
  &\,\le\, \big(2 + \gamma\,C(\rho,\psi)\big)\,\|f\|_\calW^2
 \qquad\mbox{for all}\quad f \in \calW,
\end{align}
where $C(\rho,\psi)$ is defined in \eqref{eq:C-def}.
\end{lemma}

\begin{proof}
For any $f\in\calW$, we use the reproducing property to write
\begin{align*}
  \|f\|_{\calL^2_\rho}^2 \,&=\,\int_{-\infty}^\infty |f(y)|^2 \,\rho(y)\,\rd y
  \,=\, \int_{-\infty}^\infty |\langle f, K(\cdot,y)\rangle_\calW|^2 \,\rho(y)\,\rd y \\
  &\,\le\, \int_{-\infty}^\infty \|f\|_\calW^2\,\|K(\cdot,y)\|_\calW^2 \,\rho(y)\,\rd y
 \,=\, \| f\|_{\calW}^2 \int_{-\infty}^\infty K(y, y)\, \rho(y) \,\rd y,
\end{align*}
where we used the symmetry and the reproducing property of the kernel to
write $\|K(\cdot,y)\|_\calW^2 = \langle K(\cdot, y), K(\cdot,
y)\rangle_{\calW} = K(y,y)$.

Starting from the formula \eqref{eq:eta1Da}, we use the Fubini theorem to
change the order of integration, to obtain
\begin{align}
\label{eq:int-eta-yy}
&\int_{-\infty}^\infty \eta(y, y)\, \rho(y) \,\rd y
\,=\, \int_{-\infty}^\infty \bigg(\int_{-\infty}^{y} \frac{(\Phi(t))^2}{\psi(t)} \,\rd t
 + \int_{y}^\infty \frac{(1 - \Phi(t))^2}{\psi(t)} \,\rd t\bigg) \rho (y) \,\rd y
 \nonumber\\
&\,=\, \int_{-\infty}^\infty \frac{(\Phi(t))^2}{\psi(t)} \bigg(\int_{t}^{\infty} \rho (y) \,\rd y \bigg) \,\rd t
 + \int_{-\infty}^\infty \frac{(1 - \Phi(t))^2}{\psi(t)} \bigg(\int_{-\infty}^t \rho (y) \,\rd y \bigg) \,\rd t
 \nonumber\\
&\,=\, \int_{-\infty}^\infty \frac{(\Phi(t))^2}{\psi(t)} (1-\Phi(t)) \,\rd t
 + \int_{-\infty}^\infty \frac{(1 - \Phi(t))^2}{\psi(t)} \Phi(t) \,\rd t
 \nonumber\\
&\,=\, \int_{-\infty}^\infty \frac{\Phi(t)\,(1-\Phi(t))}{\psi(t)} \,\rd t
\nonumber\\
&\,=\, C(\rho,\psi),
\end{align}
which is finite under the assumption \eqref{eq:C-def}, or equivalently
\eqref{eq:strong}. Hence
\[
  \int_{-\infty}^\infty K(y, y)\, \rho(y) \,\rd y \,=\, 1 + \gamma\,C(\rho,\psi),
\]
which leads to the first bound. The second bound then follows from
$\|f\|_\calH^2 \le \|f\|_{\calL^2_\rho}^2 + \|f\|_\calW^2$, which follows
from the definitions \eqref{eq:Hnorm1D2} and \eqref{eq:Wnorm1D2}.
\end{proof}

Since we have now proved the norm equivalence of $\calH$ and $\calW$ under
the stronger condition~\eqref{eq:strong}, and since $\calW$ is a RKHS, it
follows that $\calH$ under the same condition is another RKHS, but not one
with a known simple kernel that corresponds to the inner product
\eqref{eq:inner1DH}. Since the spaces $\calH$ and $\calW$ are equivalent,
it makes sense that from now on we choose to work with the inner product
\eqref{eq:inner1DW}, and use interchangeably the names $\calH$ and $\calW$
for the space itself.

Finally we restate the main result of this subsection, that the $\calH$
and $\calW$ norms are equivalent under the condition \eqref{eq:strong},
but now with an improved embedding constant. We do not know whether the
embedding property holds (with a different embedding constant) under the
weaker condition \eqref{eq:weak} or some other intermediate condition.

\begin{theorem} \label{thm:equiv1D}
Under the condition \eqref{eq:strong}, the spaces $\calH$ and $\calW$ are
equivalent, with
\begin{equation} \label{eq:equiv1Dnew}
 \|f\|_{\calW}^2 \,\leq\, \|f\|_\calH^2
 \,\leq\, \big(1 + \gamma\,C(\rho,\psi) \big)\,\|f\|_{\calW}^2
 \qquad\mbox{for all}\quad f \in \calH,
\end{equation}
where $C(\rho,\psi)$ is defined in \eqref{eq:C-def}. This space is a
reproducing kernel Hilbert space and, when equipped with the inner
product \eqref{eq:inner1DW}, its kernel is given by \eqref{eq:ker1D}.
\end{theorem}

\begin{proof}
As discussed at the start of this section the first inequality follows by
the Cauchy--Schwarz inequality. Together with \eqref{eq:embedding1}
this implies that the spaces $\calH$ and $\calW$ are equivalent.  We also
know from the preceding subsection that for every $f\in\calW$, or
equivalently $f\in\calH$, we can write uniquely
\[
  f \,=\, f_\emptyset + f_{\{1\}},
\]
where
\[
  f_\emptyset \,\coloneqq\,  I_\rho(f) \,=\, \int_{-\infty}^\infty f(x)\,\rho(x) \, \rd x
  \qquad\mbox{and}\qquad
  f_{\{1\}} \,\coloneqq\, f - f_\emptyset \in \calV.
\]
It is easily seen that the two terms $ f_{\{1\}}$ and $f_\emptyset$ are
orthogonal in $\calL^2_\rho$ (by virtue of \eqref{eq:calVdef}) and also
orthogonal in $\calW$, from which it follows that
\begin{align*}
  \|f\|_{\calL^2_\rho}^2
  &\,=\, \|f_\emptyset\|_{\calL^2_\rho}^2 + \|f_{\{1\}}\|_{\calL^2_\rho}^2
  \,=\, f_\emptyset^2 + \|f_{\{1\}}\|_{\calL^2_\rho}^2,
  \\
\|f\|_\calW^2
  &\,=\, \|f_\emptyset\|_\calW^2 + \|f_{\{1\}}\|_\calW^2
  \,=\, f_\emptyset^2 + \frac{1}{\gamma}\,\|f'_{\{1\}}\|_{\calL^2_\psi}^2,
\end{align*}
and therefore
\begin{align} \label{eq:part1}
  \|f\|_\calH^2
  \,=\, \|f\|_{\calL^2_\rho}^2 + \frac{1}{\gamma}\,\|f'\|_{\calL^2_\psi}^2
  \,&=\, f_\emptyset^2 + \|f_{\{1\}}\|_{\calL^2_\rho}^2 + \frac{1}{\gamma}\,\|f'_{\{1\}}\|_{\calL^2_\psi}^2
  \nonumber\\
  &=\, \|f\|_\calW^2 + \|f_{\{1\}}\|_{\calL^2_\rho}^2.
\end{align}
(We shall meet this double orthogonality property in more general form in
Section~\ref{sec:HD}.) From the reproducing property \eqref{eq:reproV} in
$\calV$ we have
\begin{align*}
  \big|f_{\{1\}}(y)\big|^2
  \,=\, \langle f_{\{1\}}, \eta(\cdot,y)\rangle_{\calV}^2
  \,&\le\, \|f_{\{1\}}\|_{\calV}^2\, \|\eta(\cdot,y)\|_{\calV}^2
  \\
  &=\, \|f_{\{1\}}\|_{\calV}^2\, \eta(y,y)
  \,\le\, \gamma\,\|f\|_\calW^2\, \eta(y,y),
\end{align*}
and therefore
\begin{align} \label{eq:part2}
  \|f_{\{1\}}\|_{\calL^2_\rho}^2
  \,&=\, \int_{-\infty}^\infty \big|f_{\{1\}}(y)\big|^2\,\rho(y)\,\rd y
  \nonumber\\
  &\le\, \gamma\,\|f\|_\calW^2\, \int_{-\infty}^\infty \eta(y,y)\,\rho(y)\,\rd y
  \,=\, \gamma\,C(\rho,\psi)\,\|f\|_\calW^2.
\end{align}
Inserting \eqref{eq:part2} into \eqref{eq:part1} gives the required
result.
\end{proof}

\begin{remark}\label{rem:embed-constant}
Although they do not consider the spaces $\calH$ and $\calW$ explicitly,
in \cite{HefRit15} the setting is general enough that it covers the
equivalence \eqref{eq:equiv1Dnew}. However, the constants in
\cite{HefRit15} are $\big(1 + \sqrt{\gamma}\, C(\rho, \psi) + \gamma\,
(C(\rho, \psi))^2\big)^{-1}$ on the left and $1 + \sqrt{\gamma}\, C(\rho,
\psi)+ \gamma\, (C(\rho, \psi))^2$ on the right, which are not as sharp as
our constants of $1$ and $1 + \gamma\, C(\rho, \psi)$, respectively. Also,
in $d$ dimensions they only consider tensor product spaces and so the
remainder of this paper is not covered by \cite{HefRit15}.
\end{remark}

\section{The $d$-dimensional case} \label{sec:HD}

Let $d\ge 1$. In this section we will define the Sobolev space $\calH_d$
and the ANOVA space $\calW_d$ in $d$ dimensions, and then show that their
norms are equivalent under the condition \eqref{eq:strong}. Starting
with the set of locally integrable functions on $\R^d$, we define the
spaces $\calH_d$ and $\calW_d$ to be the restriction of such functions for which
the norms \eqref{eq:H-norm} and \eqref{eq:W-norm} are finite,
respectively. Both $\calH_d$ and $\calW_d$ are Hilbert spaces, and rather
than restating the norms, here we give the corresponding inner products:
\begin{align}
 \langle f, \widetilde{f} \rangle_{\calH_d}
 \,\coloneqq\,
 &\sum_{\setu \subseteq \setD} \frac{1}{\gamma_\setu}
 \int_{\R^d} \partial^\setu f(\bsx)\,\partial^\setu \widetilde{f}(\bsx)\,
 \bspsi_\setu(\bsx_\setu)\, \bsrho_{\setD\setminus\setu}(\bsx_{\setD\setminus\setu}) \,\rd\bsx, \nonumber
 \\
  \label{eq:W-inner}
 \langle f, \widetilde{f} \rangle_{\calW_d}
 \,\coloneqq\,
 &\sum_{\setu \subseteq \setD} \frac{1}{\gamma_\setu}
 \int_{\R^{|\setu|}} \bigg(\int_{\R^{d - |\setu|}} \partial^\setu f(\bsx_\setu,\bsx_{\setD\setminus\setu})\,
 \bsrho_{\setD\setminus\setu}(\bsx_{\setD\setminus\setu})
 \,\rd\bsx_{\setD\setminus\setu} \bigg)
\\\nonumber
&\qquad \cdot \bigg(\int_{\R^{d- |\setu|}} \partial^\setu \widetilde{f}(\bsx_\setu, \bsx_{\setD\setminus\setu})\,
 \bsrho_{\setD\setminus\setu}(\bsx_{\setD\setminus\setu})
 \,\rd\bsx_{\setD\setminus\setu}
 \bigg)
 \bspsi_\setu(\bsx_\setu)\,\rd \bsx_\setu.
\end{align}

In the following, we explain the various ingredients in detail. First, we
give the notation for mixed partial derivatives of first order. Let
\[
 \partial^j \,\coloneqq\, \pd{}{}{x_j},
\]
and for a subset $\setu \subseteq \setD\coloneqq \{1,2, \ldots,d\}$, let
$\partial^\setu$ denote the first-order mixed partial derivative with
respect to the variables $x_j$ for $j\in\setu$, given by
\[
 \partial^\setu \,\coloneqq\, \prod_{j \in \setu} \pd{}{}{x_j}
 \,=\, \prod_{j \in \setu} \partial^j.
\]
These derivatives should be understood as \emph{weak derivatives}: for
example the $\setu$-th weak derivative of $f$ is the locally integrable function
$\partial^\setu f : \R^d \to \R$ satisfying
\begin{equation} \label{eq:weak_D}
 \int_{\R^d} \partial^\setu f(\bsx)\, v(\bsx) \,\rd \bsx
 \,=\, (-1)^{|\setu|} \int_{\R^d} f(\bsx)\, \partial^\setu v(\bsx) \,\rd \bsx
 \qquad\text{for all}\quad v \in \calC^\infty_0(\R^d),
\end{equation}
where $|\setu|$ is the cardinality of $ \setu$, and $\calC^\infty_0(\R^d)$
is the space of infinitely differentiable functions on $\R^d$ with compact
support.
Since the weak derivative, if it exists, is assumed to be locally 
integrable, it follows that it is equivalent to a function for which 
point evaluation is well defined almost everywhere on $\R^d$. 
Hence, in each term in the $\calW$ norm \eqref{eq:W-norm} and inner product \eqref{eq:W-inner}
the inner integral over $\R^{d - |\setu|}$ is well defined for almost all $\bsx_\setu \in \R^{|\setu|}$. 
Then since we eventually integrate with respect to 
$\bsx_\setu \in \R^{|\setu|}$, both the norm and inner product in $\calW$ are well defined.

For each $j=1,\ldots,d$, let $\rho_j: \R \to \R_+$ be a strictly positive
probability density function defined on~$\R$ with the corresponding
distribution function denoted by $\Phi_j$, and let $\psi_j: \R \to \R_+$
be a locally integrable strictly positive function on~$\R$ such that
$1/\psi_j$ is also locally integrable. Both spaces $\calH_d$ and $\calW_d$
involve weak derivatives and weight functions, with
differentiated variables weighted differently to undifferentiated
variables. For each $\setu\subseteq\setD$, the ``active variables'' (the
differentiated variables) are weighted by the product
\[
 \bspsi_\setu(\bsx_\setu) \,\coloneqq\, \prod_{j \in \setu} \psi_j(x_j),
 \qquad\mbox{with}\quad \bsx_\setu\coloneqq\{x_j: j\in\setu\},
\]
while the ``inactive variables'' are weighted by
\[
 \bsrho_{\setD\setminus\setu}(\bsx_{\setD\setminus\setu})
 \,\coloneqq\, \prod_{j\in\setD\setminus\setu} \rho(x_j),
 \qquad\mbox{with}\qquad \bsx_{\setD\setminus\setu} \,\coloneqq\, \{x_j:j\in\setD\setminus\setu\}.
\]

Both $\calH_d$ and $\calW_d$ also involve weight parameters: there is a
weight parameter $\gamma_\setu>0$ for each $\setu\subseteq\setD$, and
together they moderate the relative contribution of the $2^d$ derivatives
to the norm. For a function $f$ in the unit ball, if $\gamma_\setu$ is
small then the derivative $\partial^\setu f$ must contribute less to the
norm. In the limiting case where $\gamma_\setu = 0$, we assume that
$\partial^\setu f\equiv 0$.

Using the Cauchy--Schwarz inequality and the fact that each $\rho_j$ is a
probability density, we have
\begin{align*}
 &\bigg|\int_{\R^{d - |\setu|}} \partial^\setu f(\bsx_\setu, \bsx_{\setD\setminus\setu})\,
 \bsrho_{\setD\setminus\setu}(\bsx_{\setD\setminus\setu})
 \,\rd\bsx_{\setD\setminus\setu} \bigg|^2
\\
 &\le\, \int_{\R^{d - |\setu|}} \big|\partial^\setu f(\bsx_\setu, \bsx_{\setD\setminus\setu})\big|^2\,
 \bsrho_{\setD\setminus\setu}(\bsx_{\setD\setminus\setu})
 \,\rd\bsx_{\setD\setminus\setu}.
\end{align*}
Thus $\|f\|_{\calW_d}^2 \le \|f\|_{\calH_d}^2$ and we conclude trivially
that $\calH_d$ is embedded in $\calW_d$.

The $\setu = \emptyset$ terms in the norm for $\calH_d$ and $\calW_d$
correspond to, respectively,
\[
 \|f\|_{\calL^2_{\bsrho}}^2 \,\coloneqq\,
 \int_{\R^d} \big| f(\bsx)\big|^2\, \bsrho(\bsx) \,\rd\bsx
  \qquad\mbox{and}\qquad
  \big|I_\bsrho(f)\big|^2
  \,\coloneqq\, \bigg|\int_{\bbR^d} f(\bsx)\,\bsrho(\bsx)\,\rd\bsx\bigg|^2,
\]
with $\bsrho(\bsx) \coloneqq \prod_{j=1}^d \rho_j(x_j)$. This is consistent with
the representations \eqref{eq:Hnorm1D2}--\eqref{eq:Wnorm1D2} in one
dimension. Trivially we know that $\calH_d$ is embedded in
$\calL^2_\bsrho$. But $\calW_d$ may or may not be embedded in
$\calL^2_\bsrho$.

We follow the general strategy for the $1$-dimensional case to obtain norm
equivalence for $\calH_d$ and $\calW_d$. In Subsection~\ref{sec:kernel} we
verify that the reproducing kernel for $\calW_d$ exists under the weaker
condition \eqref{eq:weak} for the weight functions $(\rho_j,\psi_j)$ for
all $j=1,\ldots,d$. Again this case was not previously considered in
\cite{NK14}, and moreover, some technical details were not adequately
addressed in \cite{NK14}. In Subsection~\ref{sec:ANOVA} we discuss the
ANOVA decomposition which provides the crucial steps to prove norm
equivalence in Subsection~\ref{sec:equiv} under the the stronger condition
\eqref{eq:strong} for all $(\rho_j,\psi_j)$.

\subsection{The reproducing kernel for $\calW_d$ exists under the weaker condition
\eqref{eq:weak}} \label{sec:kernel}

To verify the reproducing property of the kernel for $\calW_d$, we need a
multivariate extension of the property of \emph{absolute continuity} in
one dimension. An important property of the classical Sobolev spaces is
\emph{absolute continuity along lines}, which is the property that
functions in the first-order space $W^{1, 1}(\Omega) \coloneqq
\{f\in\calL^1(\Omega) : \partial^k f \in\calL^1(\Omega) \mbox{ for all }
k=1,\ldots,d\}$, for $\Omega \subseteq \R^d$ (with no weight functions or
weight parameters), are absolutely continuous along almost all line
segments parallel to the coordinate axes (see, e.g.,
\cite[Theorem~1.1.3/1]{Mazya11}). We show  here that this absolute
continuity along lines property also holds for functions in $\calW_d$. We
write $\bsx_{-k}\coloneqq\bsx_{\setD\setminus\{k\}}$ for brevity.

\begin{lemma}\label{lem:abs_cont}
Suppose that $f \in \calW_d$ and the condition \eqref{eq:weak} holds for all pairs
of weight functions $(\rho_j, \psi_j)$ for $j = 1, 2, \ldots, d$. Let $k \in \setD$.
Then for almost all $\bsx_{-k} \in \R^{d - 1}$
the univariate function $f(\cdot, \bsx_{-k}): \R \to \R$ is absolutely continuous on any compact interval.
\end{lemma}

\begin{proof}
The proof follows the proof of \cite[Sec.~1.1.3, Theorem 1]{Mazya11} for
classical first order Sobolev spaces $W^{1, 1}$ almost exactly. We include
it here for completeness, and because of the added difficulties caused by
$\calW_d$ involving weight functions and dealing explicitly with an unbounded domain.

Let $[a, b] \subset \R$ be a compact interval (i.e., $-\infty < a < b <
\infty$). We first show that $f \in \calW_d$ implies that the univariate
function $\partial^k f(\cdot, \bsx_{-k})$ is integrable on
$[a,b]$ for almost all $\bsx_{-k} \in \R^{d - 1}$. First consider the
following integral, which we bound using Fubini's Theorem and the
Cauchy--Schwarz inequality,
\begin{align*}
&\bigg|\int_a^b\int_{\R^{d - 1}} \partial^k f(x_{k}, \bsx_{-k})\, \bsrho_{-k}(\bsx_{-k}) \,
\rd \bsx_{-k} \,\rd x_k \bigg|
\\
&\leq\,
\bigg(\int_a^b \frac{\rd x_k}{\psi_k(x_k)}\bigg)^{1/2}
\bigg(\int_a^b \bigg|\int_{\R^{d- 1}}
\partial^k f(x_k, \bsx_{-k}) \bsrho_{-k}(\bsx_{-k}) \, \rd \bsx_{-k}\bigg|^2
\psi_k(x_k) \, \rd x_k\bigg)^{1/2}
\\
&\leq\,
\bigg(\int_a^b\frac{\rd x_k}{\psi_k(x_k)}\bigg)^{1/2}
\bigg(\int_{-\infty}^\infty \bigg|\int_{\R^{d- 1}}
\partial^k f(x_k, \bsx_{-k}) \bsrho_{-k}(\bsx_{-k}) \, \rd \bsx_{-k}\bigg|^2
\psi_k(x_k) \, \rd x_k\bigg)^{1/2}
\\
&\leq\,
\bigg(\int_a^b \frac{\rd x_k}{\psi_k(x_k)}  \bigg)^{1/2}
    \gamma_{\{k\}}^{1/2}\,  \|f\|_{\calW_d}
\,<\, \infty,
\end{align*}
where the integral over $[a, b]$ is finite because $1/\psi_k$ is locally integrable.

Equivalently, we also have absolute integrability over $\R^{d- 1} \times [a, b]$
\[
  \int_a^b\int_{\R^{d - 1}}  \big| \partial^k f(x_{k}, \bsx_{-k})\big|\, \bsrho_{-k}(\bsx_{-k}) \,
  \rd \bsx_{-k} \,\rd x_k
 \,<\, \infty.
\]
Then, by Fubini's Theorem again
\[
\int_a^b \big| \partial^k f(x_k, \bsx_{-k})\big|\, \bsrho_{-k}(\bsx_{-k})  \, \rd x_k
\,<\, \infty
\quad \text{for almost all } \bsx_{-k} \in \R^{d - 1},
\]
which, since $\bsrho_{-k}(\bsx_{-k}) > 0$ is independent of $x_k$, in
turn implies that
\[
\int_a^b \big| \partial^k f(x_k, \bsx_{-k})\big|  \, \rd x_k
\,<\, \infty
\quad \text{for almost all } \bsx_{-k} \in \R^{d - 1}.
\]
Thus, $\partial^k f(\cdot, \bsx_{-k})$ is integrable on $[a, b]$ for
almost all $\bsx_{-k} \in \R^{d - 1}$.

Now for $\bsx_{-k}\in\bbR^{d-1}$ define the function
$h(\cdot,\bsx_{-k}) \in \calL^1[a, b]$ by
\[
h(t,\bsx_{-k}) \,\coloneqq\, \int_a^{t} \partial^k f(x_k, \bsx_{-k}) \, \rd x_k,
\]
which for almost all $\bsx_{-k} \in \R^{d - 1}$ is absolutely continuous
on $[a, b]$ because $\partial^k f(\cdot, \bsx_{-k})$ is integrable on $[a,
b]$. Differentiating $h(\cdot, \bsx_{-k})$, we see that the classical
derivative $(\partial/\partial t) h(\cdot, \bsx_{-k})$ is equal to the
weak derivative $\partial^k f(\cdot, \bsx_{-k})$ almost
everywhere on $[a, b]$.

In the remainder of the proof we show that $h(\cdot, \bsx_{-k})$ and
$f(\cdot, \bsx_{-k})$ differ only by a constant on $[a, b]$, which proves
the result. Let $\{\chi_\ell\}_{\ell = 1}^\infty$ be a sequence in
$C^\infty_0[a, b]$, which shall be specified later. Since
$h(\cdot,\bsx_{-k})$ is absolutely continuous and each $\chi_\ell$ has
compact support in $[a, b]$, by integration by parts
\[
  \int_a^b h(t,\bsx_{-k})\, \dd{}{}{t}\chi_\ell(t) \, \rd t
  \,=\, - \int_a^b \pd{}{}{t}h(t,\bsx_{-k})\, \chi_\ell(t) \, \rd t.
\]
Multiplying by an arbitrary $\xi \in C^\infty_0(\R^{d - 1})$ then
integrating over $\R^{d - 1}$ gives
\begin{align}
\label{eq:abs_cont1}
  &\int_{\R^{d - 1}}\int_a^b h(t,\bsx_{-k})\, \dd{}{}{t}\chi_\ell(t)\, \xi(\bsx_{-k}) \, \rd t\, \rd \bsx_{-k}
  \nonumber\\
  &=\, - \int_{\R^{d - 1}} \int_a^b \pd{}{}{t}h(t,\bsx_{-k})\, \chi_\ell(t)\, \xi(\bsx_{-k}) \, \rd t\,\rd \bsx_{-k},
\end{align}
which is finite since $\xi$ has compact support.

By the definition of the weak derivative \eqref{eq:weak_D} we also have
\[
  \int_{\R^d} f(\bsx)\, \dd{}{}{x_k} \chi_\ell(x_k)\, \xi(\bsx_{-k}) \, \rd \bsx
 \,=\, -\int_{\R^d} \partial^k f(\bsx)\, \chi_\ell(x_k)\, \xi(\bsx_{-k}) \, \rd \bsx,
\]
which, since $\chi_\ell$ has compact support in $[a, b]$, is equivalent to
\begin{align}
\label{eq:abs_cont2}
 &\int_{\R^{d -1}}\int_a^b f(t, \bsx_{-k})\, \dd{}{}{t} \chi_\ell(t)\, \xi(\bsx_{-k}) \, \rd t\,\rd\bsx_{-k}
 \nonumber\\
& =\, -\int_{\R^{d - 1}} \int_a^b \pd{}{}{t} h(t,\bsx_{-k})\, \chi_\ell(x_k)\, \xi(\bsx_{-k}) \, \rd t\,\rd \bsx_{-k},
\end{align}
where we have relabelled $x_k$ as $t$, and on the right hand side we have
also used the property that $(\partial/\partial t)h(\cdot,\bsx_{-k}) =
\partial^k f(\cdot, \bsx_{-k})$ for almost all $\bsx_{-k}
\in \R^{d - 1}$.

Subtracting \eqref{eq:abs_cont1} from \eqref{eq:abs_cont2} and using the
fact that $\xi$ was arbitrary we have
\[
 \int_a^b \big(f(t, \bsx_{-k}) - h(t,\bsx_{-k}) \big)\, \dd{}{}{t} \chi_\ell(t) \, \rd t \,=\, 0
 \quad \text{for almost all } \bsx_{-k} \in \R^{d - 1}.
\]
Then, applying \cite[Sec.~1.1.3, Lemma]{Mazya11}, scaled to the interval
$[a, b]$ and choosing $\{\chi_\ell\}$ as in this lemma, we have that for
almost all $\bsx_{-k} \in \R^{d - 1}$ there is a constant $c(\bsx_{-k})
\in \R$ such that
\[
f(t, \bsx_{-k}) \,=\, h(t,\bsx_{-k}) + c(\bsx_{-k}) \quad \text{for almost all } t \in [a, b].
\]
Hence, $f(\cdot, \bsx_{-k})$ is absolutely continuous on $[a, b]$, (or more specifically, is
equivalent to an absolutely continuous function on $[a,b]$).
\end{proof}

The next result we need is that the integrability of functions in
$\calW_d$ necessarily implies the following limits as one variable tends
to $\pm \infty$. The proof relies on the absolute continuity property
above.

\begin{lemma} \label{lem:lims}
Suppose that $f \in \calW_d$ and the condition \eqref{eq:weak} holds for
all pairs of weight functions $(\rho_j,\psi_j)$ for $j=1,\ldots d$. Then
for any $k\in\setD$,
\begin{align}
\label{eq:lim-a}
\lim_{x_k \to - \infty} f(\bsx)\,\Phi_k(x_k) \,&=\, 0
\qquad \text{for almost all } \bsx_{-k} \in \R^{d-1}, \qquad\mbox{and}\\
\label{eq:lim-b}
\lim_{x_k \to + \infty} f(\bsx)\,(1 - \Phi_k(x_k)) \,&=\, 0
\qquad \text{for almost all } \bsx_{-k} \in \R^{d-1}.
\end{align}
\end{lemma}

\begin{proof}
Let $k\in\setD$. For notational convenience, we define $F_k :\R^d \to \R$ by
\[
F_k(\bsx) \,\coloneqq\, f(\bsx)\,\Phi_k(x_k).
\]
By Lemma~\ref{lem:abs_cont}, for almost all $\bsx_{-k} \in \R^{d -
1}$, $f(\cdot, \bsx_{-k})$ is absolutely continuous on any compact
interval. Clearly, $\Phi_k$ is absolutely continuous, and so it follows
that $F_k(\cdot, \bsx_{-k})$ is absolutely continuous on any compact
interval for almost all $\bsx_{-k} \in \R^{d -1 }$.

Now we follow a similar strategy to the proof of the one-dimensional case in
Lemma~\ref{lem:lim1D} and prove the first limit \eqref{eq:lim-a} by
contradiction. The second limit \eqref{eq:lim-b} then follows by an analogous argument.

First, we show that for any finite $c\in\bbR$ the function
$\partial^kF_k(\cdot, \bsx_{-k})$ is integrable on $(-\infty, c]$ for
almost all $\bsx_{-k} \in \R^{d-1}$. Then using the fact that
$F_k(\cdot, \bsx_{-k})$ is absolutely continuous, we show that the
premise that $\lim_{x_k \to -\infty} F_k(\bsx) \neq 0$ leads to a
contradiction.

Since $\Phi_k$ is the distribution function of $\rho_k$, the derivative of
$F_k$ with respect to $x_k$ is
\[
\partial^k F_k(\bsx) \,=\, f(\bsx)\, \rho_k(x_k) + \partial^k f(\bsx)\, \Phi_k(x_k).
\]
Consider first the following integral over $\R^{d-1} \times (-\infty,c]$
\begin{align*}
I \,\coloneqq&\,\, \bigg|\int_{\R^{d-1}} \int_{-\infty}^c \partial^k F_k(\bsx)\, \bsrho_{-k}(\bsx_{-k}) \,\rd x_k \rd \bsx_{-k}\bigg|
\\
\leq&\,
\underbrace{
\bigg|\int_{\R^{d-1}} \int_{-\infty}^c f(\bsx)\, \rho_k(x_k)\,\bsrho_{-k}(\bsx_{-k})\,\rd x_k\rd \bsx_{-k}\bigg|}_{I_1}
\\ &\qquad+
\underbrace{
\bigg|\int_{\R^{d-1}} \int_{-\infty}^c \partial^k f(\bsx)\, \Phi_k(x_k)\, \bsrho_{-k}(\bsx_{-k})\,\rd x_k \rd
\bsx_{-k}\bigg|}_{I_2},
\end{align*}
where we have used the triangle inequality. Since $\R^{d - 1} \times
(-\infty, c]\subset \R^{d}$, the finiteness of $|I_\bsrho(f)|$ also
implies that $I_1 < \infty$. For $I_2$ we can swap the order of the
integrals, multiply and divide by $(\psi_k(x_k))^{1/2}$, and then use the
Cauchy--Schwarz inequality to give the bound
\begin{align*}
 &I_2
\leq
 \int_{-\infty}^c \bigg|\int_{\R^{d-1}} \partial^k f(\bsx)\bsrho_{-k}(\bsx_{-k})\rd \bsx_{-k}\bigg|\,
 (\psi_k(x_k))^{1/2}\,\frac{\Phi_k(x_k)}{(\psi_k(x_k))^{1/2}} \,\rd x_k
 \\
&\leq \bigg( \int_{-\infty}^c
\bigg|\int_{\R^{d-1}} \!\partial^k f(\bsx)\bsrho_{-k}(\bsx_{-k})\rd\bsx_{-k}\bigg|^2
 \psi_k(x_k)\,\rd x_k \bigg)^{1/2}
 \bigg(\int_{-\infty}^c \!\frac{(\Phi_k(x_k))^2}{\psi_k(x_k)}\,\rd x_k \bigg)^{1/2}
\\
&<\,\infty,
\end{align*}
where finiteness for the two factors follows from $f \in \calW_d$ and the
condition \eqref{eq:weak}, respectively. Hence $I<\infty$, and since
$\bsrho_{-k} > 0$ we also have
\[
\int_{\R^{d-1}} \int_{-\infty}^c |\partial^k F_k(\bsx)|\, \bsrho_{-k}(\bsx_{-k})
\,\rd x_k \,\rd \bsx_{-k} \,<\, \infty.
\]
It then follows by Fubini's Theorem that
\[
\int_{-\infty}^c |\partial^k F_k(\bsx)|\bsrho_{-k}(\bsx_{-k})   \, \rd x_k
\quad \text{for almost all } \bsx_{-k} \in \R^{d - 1},
\]
and since $\bsrho_{-k} > 0$ is independent of $x_k$ we also have
\[
\int_{-\infty}^c |\partial^k F_k(\bsx)|\, \rd x_k \,<\, \infty
\quad \text{for almost all } \bsx_{-k} \in \R^{d - 1},
\]
so that $\partial^k F_k(\cdot, \bsx_{-k})$ is integrable on $(-\infty,
c]$.

Let now $\bsx_{-k} \in \R^{d-1}$ be such that $\partial^k F_k(\cdot,
\bsx_{-k})$ is integrable and $F_k(\cdot, \bsx_{-k})$ is absolutely
continuous, and suppose for a contradiction that $\lim_{t \to -\infty}
F_k(t, \bsx_{-k}) \neq 0$. Then there exists $\delta > 0$, $M \in \N$ and
a sequence $\{t_m\} \subset \R$ with $t_m \to -\infty$ as $m \to \infty$,
such that
\begin{equation*}
 |F_k(t_m, \bsx_{-k})| \,\geq\, \delta
\qquad \text{for all } m \geq M.
\end{equation*}
Assume also that $t_m \le c$ for all $m \geq M$.

For any $m \geq M$, let $x_k \in (-\infty, t_m)$ be arbitrary. Then since
$F_k(\cdot, \bsx_{-k})$ is absolutely continuous on $[x_k, t_m]$, by the
Fundamental Theorem of Calculus
\[
F_k(\bsx) \,=\, F_k(t_m, \bsx_{-k}) - \int_{x_k}^{t_m} \partial^k F_k(t, \bsx_{-k}) \,\rd t.
\]
The reverse triangle inequality then gives the lower bound
\[
|F_k(\bsx)| \,\geq\, |F_k(t_m, \bsx_{-k})| - \bigg|\int_{x_k}^{t_m} \partial^k F_k(t, \bsx_{-k}) \,\rd t\bigg|
\,\geq\, \delta - \int_{-\infty}^{t_m} |\partial^k F_k(t, \bsx_{-k})| \,\rd t .
\]
Since $\partial^k F_k(\cdot, \bsx_{-k})$ is integrable on $(-\infty, c]$,
we now choose $m \geq M$ such that
\[
\int_{-\infty}^{t_m} |\partial^k F_k(t, \bsx_{-k}) \,\rd t \,\leq\, \frac{\delta}{2}.
\]
Hence, we have the lower bound $|F_k(\bsx)| \geq \delta/2$ for all $x_k
\le t_m$, from which it follows that
\[
|f(\bsx)| \,\geq\, \frac{\delta}{2\Phi_k(x_k)} \qquad\mbox{for all}\quad x_k \le t_m.
\]

Since
\[
\bigg| \int_{\R^d} f(\bsx) \bsrho(\bsx)\, \rd \bsx \bigg|
\,\leq\, \sqrt{\gamma_\emptyset}\, \|f\|_{\calW_d} \,< \, \infty,
\]
we have $f \in\calL^1_{\bsrho}$, and hence
\begin{align*}
\norm{f}{\calL^1_{\bsrho}}
&\,\geq\,
\int_{\R^{d- 1}}\int_{-\infty}^{t_m} |f(\bsx)|\, \bsrho(\bsx) \,\rd x_k \rd \bsx_{-k}\\
&\,\geq\, \int_{\R^{d- 1}}\int_{-\infty}^{t_m} \frac{\delta}{2\Phi_k(x_k)}\,
\bsrho(\bsx) \,\rd x_k \rd \bsx_{-k} \\
&\,=\, \frac{\delta}{2} \int_{-\infty}^{t_m} \frac{\rho_k(x_k)}{\Phi_k(x_k)} \,\rd x_k
\,=\, \frac{\delta}{2} \int_{-\infty}^{t_m} \dd{}{}{x_k} [\log(\Phi_k(x_k))] \,\rd x_k.
\end{align*}
However, the last integral is again divergent, as can be seen from
\eqref{eq:div1D}.

This contradicts the fact that $f \in \calL^1_{\bsrho}$. Hence, we we must
have that
\[
\lim_{x_k \to -\infty} F_k(\bsx) \,=\, 0,
\]
as required.
\end{proof}

In the theorem below we will show that the space $\calW_d$ is a RKHS if
the condition \eqref{eq:weak} holds for all pairs of weight functions
$(\rho_j,\psi_j)$. Note that the norm \eqref{eq:W-norm} and inner product
\eqref{eq:W-inner} remain well defined whether or not the condition
\eqref{eq:weak} holds, but that the reproducing kernel is not well defined
if the condition \eqref{eq:weak} fails.

\begin{theorem}\label{thm:main}
If the condition \eqref{eq:weak} holds all pairs of weight functions
$(\rho_j,\psi_j)$ for $j=1,\ldots,d$, then the space $\calW_d$ with inner
product \eqref{eq:W-inner} is a reproducing kernel Hilbert space with
kernel
\begin{equation} \label{eq:kernel}
 K_d(\bsx,\bsy) \,\coloneqq\, \sum_{\setu\subseteq\setD} \gamma_\setu \prod_{j\in\setu} \eta_j(x_j,y_j),
\end{equation}
where
\begin{align} \label{eq:eta-k}
 \eta_j(x, y) \,\coloneqq\, &\int_{-\infty}^{\min(x, y)} \frac{(\Phi_j(t))^2}{\psi_j(t)} \,\rd t
 + \int_{\max(x,y)}^\infty \frac{(1 - \Phi_j(t))^2}{\psi_j(t)} \,\rd t
 \\\nonumber
 &- \int_{\min(x,y)}^{\max(x,y)} \frac{\Phi_j(t)(1 - \Phi_j(t))}{\psi_j(t)} \,\rd t.
\end{align}
\end{theorem}

\begin{proof}
The kernel is clearly symmetric, and it is bounded due to \eqref{eq:weak}.

To verify that $\calW_d$ is a RKHS with kernel given by \eqref{eq:kernel},
we have to show that (i) $K_d(\cdot,\bsy) \in \calW_d$ for all $\bsy \in
\R^d$, and (ii) $f(\bsy) = \langle f, K_d(\cdot,\bsy)\rangle_{\calW_d}$
for all $f \in \calW_d$ and $\bsy\in \R^d$.

First we consider the norm of $K_d(\cdot,\bsy)$ for $\bsy\in\bbR^d$:
\begin{align*}
 \|K_d(\cdot,\bsy)\|_{\calW_d}^2
 \,=\, \sum_{\setu\subseteq\setD} \frac{1}{\gamma_\setu}
 \int_{\R^{|\setu|}}
 \bigg| \int_{\R^{d-|\setu|}} \pd{|\setu|}{}{\bsx_\setu} K_d(\bsx, \bsy)
 \bsrho_{\setD\setminus\setu}(\bsx_{\setD\setminus\setu})\rd\bsx_{\setD\setminus\setu}\bigg|^2
 \bspsi_\setu(\bsx_\setu) \rd\bsx_{\setu},
\end{align*}
where we have from \eqref{eq:kernel} (with $\setu$ replaced by $\setv$)
\[
  K_d(\bsx,\bsy) \,=\, \sum_{\setv\subseteq\setD} \gamma_\setv \prod_{j\in\setv} \eta_j(x_j,y_j).
\]
For $\setu\subseteq\setD$, the weak derivative of $K_d(\bsx,\bsy)$ with
respect to $\bsx_\setu$ is
\begin{equation*}
\pd{|\setu|}{}{\bsx_\setu} K_d(\bsx, \bsy)
\,=\, \sum_{\setu\subseteq \setv \subseteq \setD}
\gamma_\setv \Bigg(\prod_{j \in \setu} \pd{}{}{x_j} \eta_j(x_j, y_j)\Bigg)
\Bigg(\prod_{j \in \setv\setminus \setu} \eta_j(x_j, y_j)\Bigg),
\end{equation*}
since if $\setu \not\subseteq \setv$ then the differentiation makes the
term vanish. In turn we have
\begin{align} \label{eq:int-diff}
 \int_{\R^{d-|\setu|}} \pd{|\setu|}{}{\bsx_\setu} K_d(\bsx, \bsy)\,
 \bsrho_{\setD\setminus\setu}(\bsx_{\setD\setminus\setu})\,\rd\bsx_{\setD\setminus\setu}
 \,=\, \gamma_\setu \prod_{j \in \setu} \pd{}{}{x_j} \eta_j(x_j, y_j),
\end{align}
since $\rho_j$ is a probability density, and if $\setv\ne\setu$ then the
property \eqref{eq:zero1D} makes the term vanish. This leads to
\begin{align*}
 \|K_d(\cdot,\bsy)\|_{\calW_d}^2
 &\,=\, \sum_{\setu\subseteq\setD} \gamma_\setu
  \prod_{j \in \setu} \int_{-\infty}^\infty \bigg(\pd{}{}{x_j} \eta_j(x_j, y_j)\bigg)^2\psi_j(x_j)\,\rd x_j
 \\
 &\,=\, \sum_{\setu\subseteq\setD} \gamma_\setu
  \prod_{j \in \setu} \eta_j(y_j,y_j) \,<\,\infty,
\end{align*}
where we used \eqref{eq:intDeta}. Hence $K_d(\cdot,\bsy)\in\calW_d$. This
completes the proof of (i).

For the reproducing property (ii), consider
\begin{align} \label{eq:reprod_a}
 \langle f, K_d(\cdot, \bsy)\rangle_{\calW_d}
 &\,=\, \sum_{\setu \subseteq\setD} \frac{1}{\gamma_\setu}
 \int_{\R^{|\setu|}} \bigg(\int_{\R^{d - |\setu|}} \partial^\setu f(\bsx)\,
 \bsrho_{\setD\setminus\setu}(\bsx_{\setD\setminus\setu})
 \,\rd\bsx_{\setD\setminus\setu} \bigg) \nonumber\\
 &\qquad \cdot
 \bigg(\int_{\R^{d- |\setu|}} \pd{|\setu|}{}{\bsx_\setu} K_d(\bsx, \bsy)\,
 \bsrho_{\setD\setminus\setu}(\bsx_{\setD\setminus\setu})
 \,\rd\bsx_{\setD\setminus\setu}
 \bigg)
 \bspsi_\setu(\bsx_\setu)\,\rd \bsx_\setu \nonumber\\
 &\,=\,
 \sum_{\setu \subseteq \setD} \int_{\R^d} \partial^\setu f(\bsx)
 \Bigg(\prod_{j \in \setu} \bigg(\pd{}{}{x_j} \eta_j(x_j, y_j)\, \psi_j(x_j)\bigg)\Bigg)
 \bsrho_{\setD\setminus\setu}(\bsx_{\setD\setminus\setu})\,\rd \bsx,
\end{align}
where we used \eqref{eq:int-diff}.

Consider now any $\setu \ne\emptyset$ and suppose $k\in\setu$.
By Fubini's Theorem and Leibniz's Theorem, we can interchange the
order of integrals and derivatives to write
\begin{align}
 J_\setu
 \coloneqq& \int_{\R^d} \partial^\setu f(\bsx) \Bigg(\prod_{j \in \setu} \bigg(\pd{}{}{x_j} \eta_j(x_j, y_j)\, \psi_j(x_j)\bigg)\Bigg)
       \bsrho_{\setD\setminus\setu}(\bsx_{\setD\setminus\setu})\,\rd \bsx
       \nonumber\\\nonumber
 =&
 \int_{\R^{d - 1}}
      \partial^{\setu\setminus\{k\}} I_k(\bsx_{-k})
      \Bigg(\prod_{j \in \setu\setminus\{k\}} \bigg( \pd{}{}{x_j} \eta_j(x_j, y_j)\, \psi_j(x_j)\bigg)\Bigg)
       \bsrho_{\setD\setminus\setu}(\bsx_{\setD\setminus\setu})\,\rd \bsx_{-k}, 
\end{align}
where
\begin{align*}
 I_k(\bsx_{-k})
 \,\coloneqq\, \lim_{R \to \infty} \int_{-R}^R \partial^k f(x_k,\bsx_{-k})\,
 \pd{}{}{x_k} \eta_k(x_k, y_k)\, \psi_k(x_k) \,\rd x_k .
\end{align*}

By Lemma~\ref{lem:abs_cont}, for almost all $\bsx_{-k} \in \R^{d - 1}$,
the function $f(\cdot,\bsx_{-k})$ is absolutely continuous on $[-R,
R]$. Using \eqref{eq:Deta} and integration by parts gives
\begin{align*}
&I_k(\bsx_{-k})
\\
&=\, \lim_{R \to \infty} \bigg(
\int_{-R}^{y_k} \partial^k f(x_k,\bsx_{-k})\, \Phi_k(x_k) \,\rd x_k
+ \int_{y_k}^{R} \partial^k f(x_k,\bsx_{-k})\, (\Phi_k(x_k) - 1) \,\rd x_k\bigg)
\\
&=\, \lim_{R \to \infty} \bigg( f(y_k,\bsx_{-k})\, \Phi_k(y_k) - f(-R,\bsx_{-k})\,\Phi_k(-R)
- \int_{-R}^{y_k}  f(x_k,\bsx_{-k})\, \rho_k(x_k) \,\rd x_k\\
&+ f(R,\bsx_{-k})\,(\Phi_k(R) - 1) - f(y_k,\bsx_{-k})\, (\Phi_k(y_k) - 1)
- \int_{y_k}^{R} f(x_k,\bsx_{-k})\, \rho_k(x_k) \,\rd x_k\bigg)
\\
&=\, f(y_k,\bsx_{-k}) - \lim_{R \to \infty} \int_{-R}^{R}
  f(x_k,\bsx_{-k})\, \rho_k(x_k) \,\rd x_k,
\end{align*}
where the boundary terms vanish as $R \to \infty$ due to
Lemma~\ref{lem:lims}. Hence, we can write
\begin{align} \label{eq:J_u_rec}
J_\setu
 = & \int_{\R^{d-1}} \!\partial^{\setu\setminus\{k\}} f(y_k,\bsx_{-k})
\Bigg(\prod_{j \in \setu \setminus \{k\}} \pd{}{}{x_j} \eta_j(x_j, y_j) \psi_j(x_j)\Bigg)
 \bsrho(\bsx_{\setD\setminus\setu}) \rd \bsx_{-k}
- J_{\setu\setminus\{k\}}.
\end{align}

Returning to the inner product \eqref{eq:reprod_a}, we split the sum
over whether $d \in \setu$ to write
\begin{align*}
\langle f, K_d(\cdot, \bsy)\rangle_{\calW_d}
\,&=\, \sum_{\setu \subseteq \setD} J_\setu
\,=\,\sum_{\setu \subseteq \{1:d - 1\}} J_{\setu \cup \{d\}}
+ \sum_{\setu \subseteq \{1:d - 1\}} J_\setu
\\
&=\,\sum_{\setu \subseteq \{1:d - 1\}} \big(J_{\setu \cup \{d\}} + J_\setu\big).
\end{align*}
Using the recursive formula \eqref{eq:J_u_rec} with $\setu$ replaced by
$\setu\cup\{d\}$ and $k$ replaced by $d$ gives
\begin{align*}
&\langle f, K_d(\cdot, \bsy)\rangle_{\calW_d}
\\
&=\, \sum_{\setu \subseteq \{1:d - 1\}} \int_{\R^{d-1}} \partial^{\setu} f(y_d,\bsx_{-d})
\Bigg(\prod_{j \in \setu} \bigg(\pd{}{}{x_j} \eta_j(x_j, y_j) \psi_j(x_j)\bigg)\Bigg)
\bsrho(\bsx_{\setD\setminus\setu}) \,\rd \bsx_{-d}.
\end{align*}
Hence we have ``reproduced'' the variable $y_d$.

Applying this procedure iteratively, we can see that the integral terms
will always cancel, until we have ``reproduced'' all of the variables
$y_d, y_{d - 1}, \ldots, y_1$, and are left only with
\[
\langle f, K_d(\cdot, \bsy)\rangle_{\calW_d} \,=\, f(\bsy),
\]
which holds for almost all $\bsy \in \R^d$.
\end{proof}

\subsection{The ANOVA decomposition} \label{sec:ANOVA}

Every $d$-variate function on $\bbR^d$ can be written as a sum of $2^d$
terms of the form
\[
  f(\bsx) \,=\, \sum_{\setu\subseteq\setD} f_\setu(\bsx_\setu),
\]
where each term $f_\setu$ depends only on the variables $\bsx_\setu$.
Obviously there are infinitely many ways to do this. One way to ensure
uniqueness of the decomposition for functions $f\in \calL^1_\bsrho$ is to
impose the ``annihilating condition'' that
\[
  \int_{-\infty}^\infty f_\setu(\bsx_\setu)\,\rho_j(x_j) \,=\, 0
  \quad\mbox{for all $j\in\setu$ whenever $\setu\ne\emptyset$},
\]
which leads to the ``ANOVA decomposition'', see e.g., \cite{KSWW10}. The
ANOVA terms can be expressed using a recursive formula, or an explicit
formula \cite{KSWW10}
\begin{equation} \label{eq:exp}
  f_\setu(\bsx_\setu)
  \,=\, \sum_{\setv\subseteq\setu} (-1)^{|\setu|-|\setv|}
  \int_{\bbR^{d - |\setv|}} f(\bsx)\, \bsrho_{\setD\setminus \setv} (\bsx_{\setD \setminus \setv})
  \,\rd\bsx_{\setD \setminus \setv}.
\end{equation}

ANOVA stands for ``ANalysis Of VAriance'' and is traditionally considered
for $\calL^2_\bsrho$ functions. If $f\in \calL^2_\bsrho$ then the ANOVA
terms are orthogonal
\[
  \int_{\bbR^d} f_\setu(\bsx_\setu)\,f_\setv(\bsx_\setu)\,\bsrho(\bsx)\,\rd\bsx \,=\, \delta_{\setu,\setv},
\]
and there is a nice decomposition for the $\calL^2_\bsrho$ norm and the
variance of $f$
\[
  \|f\|^2_{\calL^2_\bsrho} \,=\, \sum_{\setu\subseteq\setD} \|f_\setu\|^2_{\calL^2_\bsrho}
  \qquad\mbox{and}\qquad
  \sigma^2(f) \,=\, \sum_{\setu\subseteq\setD} \sigma^2(f_\setu),
\]
with $\sigma^2(f) \coloneqq I_\bsrho(f^2) - (I_\bsrho(f))^2$ and
$\sigma^2(f_\setu) = I_\bsrho(f_\setu^2)$ for $\setu\ne\emptyset$.

In the space $\calW_d$ the functions are in $\calL^1_\bsrho$ but not
necessarily in $\calL^2_\bsrho$. However, there is a nice decomposition
for the $\calW_d$ norm
\begin{equation}
\label{eq:norm-decomp}
  \|f\|^2_{\calW_d} \,=\, \sum_{\setu\subseteq\setD} \|f_\setu\|^2_{\calW_d},
\end{equation}
because the ANOVA terms are orthogonal with respect to the inner product
in $\calW_d$. This can be verified directly from the explicit formula
\eqref{eq:exp}, without the need for condition \eqref{eq:weak}. 
It also follows from the general result in \cite{KSWW10}
using the annihilating condition \eqref{eq:zero1D} of the kernel, 
which is well defined under condition \eqref{eq:weak}.

\begin{lemma}
If the condition \eqref{eq:strong} holds for all pairs of weight functions
$(\rho_j,\psi_j)$ for $j=1,\ldots,d$, then the space $\calW_d$ is embedded
in $\calL^2_\bsrho$, with
\[
  \|f\|_{\calL^2_\bsrho}^2 \,\le\,
  \bigg(\sum_{\setu\subseteq\setD} \gamma_\setu \prod_{j\in\setu} C(\rho_j,\psi_j)\bigg)\,\|f\|_{\calW_d}^2
 \qquad\mbox{for all}\quad f \in \calW_d,
\]
where $C(\rho_j,\psi_j)$ is defined in \eqref{eq:C-def}.
\end{lemma}

\begin{proof}
For any $f\in\calW_d$, we use the reproducing property to write
\begin{align*}
  \|f\|_{\calL^2_\bsrho}^2
  &\,=\,\int_{\bbR^d} |f(\bsy)|^2\,\bsrho(\bsy)\,\rd \bsy
  \,=\, \int_{\bbR^d} |\langle f, K_d(\cdot,\bsy)\rangle_{\calW_d}|^2 \,\bsrho(\bsy)\,\rd \bsy
  \\
  &\,\le\, \int_{\bbR^d} \|f\|_\calW^2\,\|K_d(\cdot,\bsy)\|_{\calW_d}^2 \,\bsrho(\bsy)\,\rd \bsy,
\end{align*}
where
\[
 \|K_d(\cdot,\bsy)\|_{\calW_d}^2
 \,=\, \langle K_d(\cdot,\bsy),K_d(\cdot,\bsy)\rangle_{\calW_d}
 \,=\, K_d(\bsy,\bsy) = \sum_{\setu\subseteq\setD} \gamma_\setu \prod_{j\in\setu} \eta_j(y_j,y_j).
\]
The result now follows from \eqref{eq:int-eta-yy}.
\end{proof}

In the next subsection we will establish the norm equivalence between
$\calW_d$ and $\calH_d$. As in the one dimensional case we will make use
of the reproducing property. However, for each term in \eqref{eq:H-norm}
we only want to ``reproduce'' the variables that are \emph{not}
differentiated.

To this end, analogously to the space $\calW_d$, for $\setu
\subseteq\setD$ we introduce the space $\calW_\setu$ which we define to be
the unanchored ANOVA space of functions on $\R^{|\setu|}$ that only depend
on the variables~$\bsx_\setu$. This space is a reproducing kernel Hilbert
space with inner product
\begin{align}
\langle g, \widetilde{g}\rangle_{\calW_\setu}
\,\coloneqq\,
\sum_{\setv \subseteq \setu} \frac{1}{\gammatilde_\setv}
\int_{\R^{|\setv|}} \bigg(&\int_{\R^{|\setu| - |\setv|}}
\partial^\setv g(\bsx_\setv, \bsx_{\setu \setminus \setv})\,
\bsrho_{\setu \setminus \setv}(\bsx_{\setu \setminus \setv}) \, \rd \bsx_{\setu \setminus \setv}\bigg)
\nonumber\\\nonumber
\cdot\bigg(&\int_{\R^{|\setu| - |\setv|}}
\partial^\setv \widetilde{g}(\bsx_\setv, \bsx_{\setu \setminus \setv})\,
\bsrho_{\setu \setminus \setv} (\bsx_{\setu \setminus \setv}) \, \rd \bsx_{\setu \setminus \setv}\bigg)
\bspsi_\setv(\bsx_\setv) \, \rd \bsx_\setv,
\end{align}
induced norm
\begin{equation*}
\|g\|_{\calW_\setu}^2
\,=\,
\sum_{\setv \subseteq \setu} \frac{1}{\gammatilde_\setv}
\int_{\R^{|\setv|}} \bigg|\int_{\R^{|\setu| - |\setv|}}
\partial^\setv g(\bsx_\setv, \bsx_{\setu \setminus \setv})\,
\bsrho_{\setu \setminus \setv}(\bsx_{\setu \setminus \setv}) \, \rd \bsx_{\setu \setminus \setv}
\bigg|^2
\bspsi_\setv(\bsx_\setv) \, \rd \bsx_\setv,
\end{equation*}
and kernel $K_{\setu} : \R^{|\setu|} \times \R^{|\setu|} \to \R$ given by
\begin{equation*}
K_{\setu}(\bsx_\setu,\bsy_\setu) \,\coloneqq\, \sum_{\setv \subseteq \setu} \gammatilde_\setv
\prod_{j \in \setv} \eta_j(x_j, y_j).
\end{equation*}
Note that the case $\setu = \setD$ gives our original space $\calW_d$.

Analogous to the space $\calW_d$, every $g \in \calW_\setu$ admits an
ANOVA decomposition
\begin{equation*}
  g(\bsx_\setu) \,=\, \sum_{\setv\subseteq \setu } g_\setv(\bsx_\setv),
\end{equation*}
where each $g_\setv \in \calW_{\setv}$ depends only on the variables
$\bsx_\setv$, and is given explicitly by
\begin{equation}
\label{eq:explicit}
g_\setv(\bsx_\setv)
  \,=\, \sum_{\setw\subseteq\setv} (-1)^{|\setv|-|\setw|}
  \int_{\bbR^{|\setu| - |\setw|}} g(\bsx_\setw,\bsx_{\setu  \setminus \setw})
  \bsrho_{\setu \setminus \setw} (\bsx_{\setu \setminus \setw})\,\rd\bsx_{\setu \setminus \setw}.
\end{equation}
The ANOVA decomposition is orthogonal in $\calW_\setu$, and orthogonal in
$\calL^2_{\bsrho_\setu}$ provided $\calW_\setu$ is embedded in
$\calL^2_{\bsrho_\setu}$ (which is the case if \eqref{eq:strong} holds for
all pairs $(\rho_j, \psi_j)$ for $j \in \setu$):
\begin{equation}
\label{eq:anova-orthog}
\|g\|_{\calW_\setu}^2 \,=\, \sum_{\setv\subseteq\setu} \|g_\setv\|_{\calW_\setu}^2,
 \qquad
\|g\|_{\calL^2_{\bsrho}}^2 \,=\,
    \sum_{\setv\subseteq\setu} \|g_\setv\|_{\calL^2_{\bsrho_\setu}}^2,
\end{equation}
and the decomposition terms also satisfy the useful property
\begin{equation}
\int_{-\infty}^\infty g_\setv(\bsx_\setv)\,\rho_j(x_j)\,\rd x_j \,=\, 0
  \quad\mbox{for any } j\in\setv\ne\emptyset .\label{eq:kill}
\end{equation}
Moreover, the property \eqref{eq:kill} implies that
\begin{align}
\label{eq:g_v_norm}
 \|g_\setv\|_{\calW_\setu}^2 \,=\,
 \|g_\setv\|_{\calW_\setv}^2 \,&=\,
 \frac{1}{\gammatilde_\setv} \int_{\R^{|\setv|}}
 |\partial^\setv g_\setv(\bsx_\setv)|^2 \bspsi_\setv(\bsx_\setv) \, \rd \bsx_\setv,
\quad \text{and}\\
\label{eq:g_v_reprod}
g_\setv(\bsy_\setv) \,&=\,
\bigg\langle g_\setv, \gamma_\setv \prod_{j \in \setv} \eta_j(\cdot, y_j)\bigg\rangle_{\calW_\setv},
\end{align}
and we have also
\begin{align} \label{eq:norm-eta}
  \Bigg\|\prod_{j\in\setv} \eta_j(\cdot,y_j)\Bigg\|_{\calW_\setv}^2
  \,=\, \frac{1}{\gamma_\setv} \prod_{j\in\setv} \eta_j(y_j,y_j).
\end{align}

\subsection{Norm equivalence in $\calH_d$ and $\calW_d$ under the stronger condition
\eqref{eq:strong}} \label{sec:equiv}

To prove the general norm equivalence we require the following technical
lemma. The basic idea is that if we differentiate some function $f \in
\calW_d$ with respect to the variables $\bsx_\setu$ and then treat those
$|\setu|$ variables as fixed, then that function also belongs to the
weighted unanchored ANOVA space in $(d - |\setu|)$ dimensions.

\begin{lemma} \label{lem:Df_in_Wu}
If $f \in \calW_d$ then for each $\setu \subseteq\setD$ and almost all
$\bsx_{\setu} \in \R^{|\setu|}$, $\partial^{\setu} f(\bsx_{\setu}, \cdot)
\in \calW_{\setD\setminus\setu}$.
\end{lemma}

\begin{proof}
For $f\in\calW_d$, $\setu\subseteq\setD$ and
$\bsx_\setu\in\bbR^{|\setu|}$, we write
\begin{equation} \label{eq:H-sum}
 \|\partial^{\setu} f(\bsx_{\setu},
\cdot)\|_{\calW_{\setD\setminus \setu}}^2 \,=\, \sum_{\setv \subseteq
\setD\setminus \setu} \frac{1}{\gamma_\setv} h_\setv(\bsx_{\setu}),
\end{equation}
where for each $\setv \subseteq \setD\setminus \setu$ we define the
function $h_\setv : \R^{|\setu|} \to \R$ by
\begin{align*}
& h_\setv(\bsx_{\setu}) \coloneqq
 \\
 &\int_{\R^{|\setv|}} \!\bigg| \!\int_{\R^{d - |\setu| - |\setv|}}
 \!\!\!\partial^{\setu \cup \setv} f(\bsx_\setu, \bsx_{\setv}, \bsx_{\setD\setminus(\setu \cup \setv)})\,
 \bsrho_{\setD\setminus(\setu \cup \setv)}(\bsx_{\setD\setminus(\setu \cup \setv)})
 \,\rd \bsx_{\setD\setminus(\setu \cup \setv)}\bigg|^2
 \bspsi_{\setv}(\bsx_{\setv}) \rd \bsx_{\setv}.
\end{align*}
We have
\begin{align*}
&\int_{\R^{|\setu|}} |h_\setv(\bsx_{\setu})\,\bspsi_{\setu}(\bsx_{\setu})| \, \rd \bsx_{\setu}
\\
&=\,
\int_{\R^{|\setu| + |\setv|}} \bigg|\int_{\R^{d - |\setu| - |\setv|}}
\partial^{\setu \cup \setv} f(\bsx_\setu, \bsx_{\setv}, \bsx_{\setD\setminus(\setu \cup \setv)})
\bsrho_{\setD\setminus(\setu \cup \setv)}(\bsx_{\setD\setminus(\setu \cup \setv)})
\,\rd \bsx_{\setD\setminus(\setu \cup \setv)}
\bigg|^2
\\
&\qquad\qquad \qquad \cdot\bspsi_{\setu \cup \setv}(\bsx_{\setu \cup \setv})
 \,\rd \bsx_{\setu \cup \setv}
\\
&\leq\, \gamma_{\setu \cup \setv}\,
\|f\|_{\calW_d}^2
\,<\,\infty,
\end{align*}
where we recognised that the integral expression corresponds to one of the
terms in the norm $\|f\|_{\calW_d}^2$, see~\eqref{eq:W-norm}. Thus
$h_\setv \bspsi_{\setu}$ is integrable on $\R^{|\setu|}$. It then follows
that $h_\setv \bspsi_{\setu} = |h_\setv \bspsi_{\setu}| <  \infty$ almost
everywhere on $\R^{ |\setu|}$, and since each $\psi_j$ is strictly
positive, we also have that $h_\setv = |h_\setv| < \infty$ almost
everywhere. Hence by \eqref{eq:H-sum}, we have that $\|\partial^{\setu}
f(\bsx_{\setu}, \cdot)\|_{\calW_{\setD\setminus \setu}}$ is finite for
almost all $\bsx_{\setu} \in \R^{ |\setu|}$.
\end{proof}

\begin{theorem}
\label{thm:equiv-d-dim}
If the condition \eqref{eq:strong} holds for all pairs of weight functions
$(\rho_j,\psi_j)$ for $j=1,\ldots,d$, then the spaces $\calH_d$ and
$\calW_d$ are equivalent, and
\begin{equation}
\label{eq:equiv-d-dim}
 \|f\|_{\calW_d}^2
 \,\leq\, \|f\|_{\calH_d}^2
 \,\leq\,
 \bigg(\max_{\setv \subseteq\setD} \sum_{\setw \subseteq \setv} \frac{\gamma_\setv}{\gamma_{\setv \setminus \setw}}
 \prod_{j \in \setw} C(\rho_j, \psi_j)\bigg)\,
 \|f\|_{\calW_d}^2,
\end{equation}
where $C(\rho_j, \psi_j)$ is defined in \eqref{eq:C-def}.
\end{theorem}

\begin{proof}
As we explained before, $\|f\|_{\calW_d}^2 \le \|f\|_{\calH_d}^2$ follows
easily from the Cauchy--Schwarz inequality.

To prove the second inequality in \eqref{eq:equiv-d-dim} we use
Lemma~\ref{lem:Df_in_Wu} to reproduce the variables that are not
differentiated. First, for $f\in\calW_d$, $\setu \subseteq\setD$ and
$\bsx_\setu \in \R^{|\setu|}$, define the function $g^\setu : \R^{d -
|\setu|} \to \R$ by
\[
  g^\setu(\bsx_{\setD\setminus\setu}) \,\coloneqq\, \partial^{\setu}f(\bsx_{\setu}, \bsx_{\setD\setminus\setu}),
\]
where to ease the notation we omit the dependence of $g^\setu$ on
$\bsx_\setu$. Then the $\calH_d$ norm \eqref{eq:H-norm} of $f$ can be
written as
\begin{align} \label{eq:H-norm_u}
  \|f\|_{\calH_d}^2
  &\,=\, \sum_{\setu\subseteq\setD}
\frac{1}{\gamma_\setu}
\int_{\R^{|\setu|}} \bigg(
\int_{\R^{d-|\setu|}} \big|\partial^{\setu} f(\bsx_{\setu}, \bsx_{\setD\setminus\setu}) \big|^2\,
\bsrho_{\setD\setminus\setu}(\bsx_{\setD\setminus\setu}) \,\rd \bsx_{\setD\setminus\setu}\bigg)
\bspsi_{\setu}(\bsx_{\setu}) \,\rd \bsx_\setu
 \nonumber\\
&\,=\, \sum_{\setu\subseteq\setD}
\frac{1}{\gamma_\setu} \int_{\R^{|\setu|}} \|g^\setu\|_{\calL^2_{\bsrho_{\setD\setminus\setu}}}^2
\bspsi_\setu(\bsx_\setu) \,\rd \bsx_\setu,
\end{align}
where $\|\cdot\|_{\calL^2_{\bsrho_{\setD\setminus\setu}}}$ is the
$\calL^2$-norm with respect to the variables $\bsx_{\setD\setminus\setu}$
and weight function $\bsrho_{\setD\setminus\setu}$.

By Lemma~\ref{lem:Df_in_Wu}, for almost all $\bsx_{\setu} \in
\R^{|\setu|}$, we have $g^\setu \in \calW_{\setD\setminus \setu}$. Each
$g^\setu$ admits an ANOVA decomposition
\begin{equation*}
g^\setu \,=\, \sum_{\setv \subseteq \setD\setminus \setu}
(g^\setu)_\setv,
\end{equation*}
where $(g^\setu)_\setv \in \calW_\setv$ and which, by \eqref{eq:explicit},
can be written explicitly in terms of $g^\setu$, and thus in terms of $f$,
as
\begin{equation}
\label{eq:huv_explicit}
(g^\setu)_\setv(\bsx_\setv) \,=\,
\sum_{\setw \subseteq \setv} (-1)^{|\setv| - |\setw|}
\int_{\R^{d - |\setu| - |\setw|}} \partial^\setu f(\bsx)\,
\bsrho_{\setD\setminus(\setu \cup \setw)} (\bsx_{\setD\setminus(\setu \cup \setw)}) \,\rd \bsx_{\setD\setminus(\setu \cup \setw)}.
\end{equation}
Since \eqref{eq:strong} holds, the ANOVA decomposition is orthogonal
in the $\calL^2$-norm (see \eqref{eq:anova-orthog}), and so we can write
\begin{equation}
\label{eq:hu-L2-norm}
\int_{\R^{|\setu|}} \|g^\setu\|_{\calL^2_{\bsrho_{\setD\setminus\setu}}}^2
\bspsi_\setu(\bsx_\setu) \, \rd \bsx_\setu
\,=\, \sum_{\setv \subseteq \setD\setminus \setu}
\int_{\R^{|\setu|}} \|(g^\setu)_\setv\|_{\calL^2_{\bsrho_\setv}}^2
\bspsi_\setu(\bsx_\setu) \, \rd \bsx_\setu.
\end{equation}

Now, using the reproducing property \eqref{eq:g_v_reprod} in $\calW_\setv$
as well as \eqref{eq:norm-eta}, for all $\bsx_\setv \in \R^{|\setv|}$ with
$\setv \subseteq \setD \setminus \setu$, we have
\begin{align*}
 \big| (g^\setu)_\setv(\bsx_\setv) \big|^2
\,=\, \bigg\langle (g^\setu)_\setv, \gamma_\setv \prod_{j \in \setv} \eta_j(\cdot, y_j)
 \bigg\rangle_{\calW_\setv}^2
&\,\leq\, \gamma_\setv^2\,\|(g^\setu)_\setv\|_{\calW_\setv}^2\,
\Bigg\|\prod_{j \in \setv} \eta_j(\cdot, y_j) \Bigg\|_{\calW_\setv}^2 \\
&\,=\, \gamma_\setv\,\|(g^\setu)_\setv\|_{\calW_\setv}^2\,\prod_{j \in \setv} \eta_j(y_j, y_j).
\end{align*}
Hence, the $\calL^2_{\bsrho_\setv}$-norm of $(g^\setu)_\setv$ is bounded by
\begin{align}
\label{eq:huv-L2-norm}
\|(g^\setu)_\setv\|_{\calL^2_{\bsrho_\setv}}^2 \,&\leq\,
\gamma_\setv\, \|(g^\setu)_\setv\|_{\calW_\setv}^2\,
\prod_{j \in \setv} \int_{-\infty}^\infty \eta_j(y_j, y_j) \rho_j(y_j) \, \rd y_j
\nonumber\\
&=\,
\gamma_\setv\, \|(g^\setu)_\setv\|_{\calW_\setv}^2\,
\prod_{j \in \setv} C(\rho_j, \psi_j),
\end{align}
where we have used \eqref{eq:int-eta-yy}.

Next we use property \eqref{eq:g_v_norm} for the $\calW_\setv$-norm of
$(g^\setu)_\setv$ and substitute in the explicit formula
\eqref{eq:huv_explicit} for $(g^\setu)_\setv$, to obtain
\begin{align}
\label{eq:huv-W-norm}
&\|(g^\setu)_\setv\|_{\calW_\setv}^2
\,=\, \frac{1}{\gamma_\setv}
\int_{\R^{|\setv|}} \big|\partial^\setv(g^\setu)_\setv(\bsx_\setv)\big|^2\, \bspsi_\setv(\bsx_\setv)
\,\rd \bsx _\setv
\\
&=\, \frac{1}{\gamma_\setv}
\int_{\R^{|\setv|}} \bigg| \partial^\setv\bigg(\int_{\R^{d - |\setu| - |\setv|}}
\partial^{\setu} f(\bsx)
\bsrho_{\setD\setminus(\setu \cup \setv)}(\bsx_{\setD\setminus(\setu \cup \setv)}) \,\rd \bsx_{\setD\setminus(\setu \cup \setv)}\bigg)
\bigg|^2 \bspsi_\setv(\bsx_\setv) \, \rd \bsx_\setv
\nonumber\\
&=\, \frac{1}{\gamma_\setv}
\int_{\R^{|\setv|}} \bigg|\int_{\R^{d - |\setu| - |\setv|}}
\partial^{\setu \cup \setv} f(\bsx)\,
\bsrho_{\setD\setminus(\setu \cup \setv)}(\bsx_{\setD\setminus(\setu \cup \setv)}) \,\rd \bsx_{\setD\setminus(\setu \cup \setv)}
\bigg|^2 \bspsi_\setv(\bsx_\setv) \, \rd \bsx _\setv,
\nonumber
\end{align}
where in the last step we applied the Leibniz rule for differentiation
under the integral from \cite[Theorem 4]{GKLS18}.

Substituting \eqref{eq:huv-W-norm} into \eqref{eq:huv-L2-norm}, and in
turn into \eqref{eq:hu-L2-norm} and then \eqref{eq:H-norm_u}, we obtain
\begin{align*}
&\|f\|_{\calH_d}^2 \,\leq\,
\sum_{\setu \subseteq \setD}\frac{1}{\gamma_\setu}
\sum_{\setv \subseteq \setD \setminus \setu}
\bigg( \prod_{j \in \setv} C(\rho_j, \psi_j) \bigg)
\\
&\cdot
\int_{\R^{|\setu| + |\setv|}} \!\bigg|\int_{\R^{d - |\setu| - |\setv|}}
 \partial^{\setu \cup \setv} f(\bsx)\,
\bsrho_{\setD\setminus(\setu \cup \setv)}(\bsx_{\setD\setminus(\setu \cup \setv)}) \,\rd \bsx_{\setD\setminus(\setu \cup \setv)}
\bigg|^2 \bspsi_{\setu \cup \setv}(\bsx_{\setu \cup \setv}) \,\rd \bsx _{\setu \cup \setv} .
\end{align*}
Substituting $\setw = \setu \cup \setv$ and then rearranging the
sums, we can write this as
\begin{align*}
 \|f\|_{\calH_d}^2
 &\,\leq\,
 \sum_{\setw \subseteq \setD} \frac{1}{\gamma_\setw}
\Bigg(  \sum_{\setv \subseteq \setw}
 \frac{\gamma_\setw}{\gamma_{\setw \setminus \setv}}
 \prod_{j \in \setv} C(\rho_j, \psi_j) \Bigg) \\
&\qquad\cdot \int_{\R^{|\setw|}} \bigg|\int_{\R^{d - |\setw|}}
 \partial^{\setw} f(\bsx)\,
 \bsrho_{\setD\setminus\setw}(\bsx_{\setD\setminus\setw}) \,\rd \bsx_{\setD\setminus\setw}
 \bigg|^2 \bspsi_{\setw}(\bsx_{\setw}) \,\rd \bsx_{\setw} \\
&\,\leq \bigg(\max_{\setw \subseteq \setD} \sum_{\setv \subseteq \setw}
\frac{\gamma_\setw}{\gamma_{\setw \setminus \setv}} \prod_{j \in \setv}
C(\rho_j, \psi_j) \bigg) \|f\|_{\calW_d}^2,
\end{align*}
where in the first step we have multiplied and divided each term by
$\gamma_\setw$ to give the correct weights. Finally interchanging the
labels for $\setv$ and $\setw$ gives the required result.
\end{proof}

\section{Analysis of QMC with preintegration for option pricing}
\label{sec:option}

To conclude this paper we return to the application that served as our
initial motivation, namely, we use the equivalence from
Theorem~\ref{thm:equiv-d-dim} to obtain a rigorous error bound for a QMC
method combined with preintegration technique for approximating the fair
price of an option.

As a concrete example we consider the \emph{arithmetic-average Asian call
option} analysed in, e.g., \cite{GKS10,GKS13,GKS17note,GKLS18}. Without
going into the details here we simply recall that the problem can be
expressed, after an appropriate change of variables, as an integral of the
form
\begin{align} \label{eq:option1}
  I_d(f) \,=\, \int_{\bbR^d} f(\bsx)\,\bsrho(\bsx)\,\rd\bsx, \qquad
  f(\bsx) \,=\, \max(\phi(\bsx),0),
\end{align}
with $\bsrho(\bsx)$ being a product of standard normal density $\rho(x) =
\frac{1}{\sqrt{2\pi}}e^{-x^2/2}$, and
\begin{align} \label{eq:option2}
  \phi(\bsx)
 \,=\,
  \frac{1}{d} \sum_{\ell = 1}^d S_0 \exp \Big( (r - \tfrac{1}{2} \sigma^2) \tfrac{\ell\, T}{d}
  + \sigma \bsA_\ell\, \bsx\Big) \,-K,
\end{align}
where $S_0$ is the initial asset price, $K$ is the strike price, $r$ is
the risk-free interest rate, $\sigma$ is the volatility, $d$ is the number
of equal time steps with final time $T$, and $\bsA_\ell$ are the rows of a
matrix $A$ arising from some factorisation of the covariance matrix
$\Sigma = [\min(iT/d,jT/d)]_{i,j\in\setD} = AA^\tr$. See, e.g.,
\cite{GKS10} for three factorisation methods (\emph{Cholesky a.k.a.\
standard construction}, \emph{Brownian bridge construction},
\emph{principal components construction (PCA)}) that lead to different
matrices $A$. The maximum appears in \eqref{eq:option1} since an option is
worthless when its value is negative. This gives rise to a \emph{kink} in
the integrand $f$ even though $\phi$ is smooth.

An $N$-point \emph{randomly shifted lattice rule} (see e.g., \cite{DKS13})
approximates $I_d(f)$ by
\[
 Q_{d, N} (f) \,=\, \frac{1}{N} \sum_{n = 0}^{N - 1} f(\bst_n),
\quad
\text{with} \quad
\bst_n \coloneqq \bsPhi^{-1}\bigg(\bigg\{ \frac{n \bsz}{N} + \bsDelta \bigg\}\bigg),
\]
where $\bsz \in \N^d$ is the \emph{generating vector}, $\bsDelta \in [0,
1)^d$ is a uniformly distributed \emph{random shift}, $\{ \cdot \}$
denotes taking the fractional part of each component in a vector and
$\bsPhi^{-1}$ denotes applying the inverse cumulative normal distribution
function $\Phi^{-1}$ to each coordinate. \emph{If} the integrand $f$
belongs to the ANOVA space $\calW_d$, then a generating vector $\bsz$ can
be constructed to achieve a RMS error bound close to $\calO(1/N)$, see
\cite[Theorem 8]{NK14}. Unfortunately, the kink means that our $f$ lacks
the smoothness requirement to be in $\calW_d$.

The \emph{smoothing by preintegration} technique from \cite{GKLS18} goes
as follows. First we integrate out one strategically chosen variable, say,
$x_1$:
\[
 P_1 f(\bsx_{2:d}) \,\coloneqq\, \int_{-\infty}^\infty f(x_1, \bsx_{2:d})\, \rho(x_1) \, \rd x_1,
\]
which is either computed analytically or numerically by a 1-dimensional
quadrature rule to high accuracy. Then we apply a randomly shifted lattice
rule to the resulting $(d-1)$-dimensional function $P_1 f$. It has been
established in \cite{GKS13,GKLS18} that $P_1 f$ belongs to the Sobolev
space $\calH_{d-1}$ for all three factorisation methods mentioned earlier.
Indeed, the general theory in these papers applies to functions of the
form $f = \max(\phi,0)$, with a number of specific conditions on the
generic function $\phi$ and the density $\rho$, including
$(\partial^{d-2}/\partial x_1^{d - 2})\phi\in \calH_d$ and $\partial^1\phi
> 0$. All of these conditions have been verified for our specific function
\eqref{eq:option2}, see \cite[Thm~3 and Sec.~6]{GKLS18}. For standard and
Brownian bridge constructions these assumptions have also been verified
for any choice of preintegration variable $x_j$, whereas for principal
components construction these are only guaranteed for the variable $x_1$.

So, on the one hand we know that $P_1 f\in \calH_{d-1}$, and on the other
hand we have established in this paper that $\calH_{d-1}$ is equivalent to
$\calW_{d-1}$. We therefore conclude that $P_1 f\in \calW_{d-1}$, and
hence we can apply \cite[Theorem 8]{NK14} to bound the error.

\begin{theorem}
For the Asian option pricing problem
\eqref{eq:option1}--\eqref{eq:option2}, an $N$-point randomly shifted
lattice rule can be constructed for the preintegrated $(d-1)$-dimensional
function $P_1 f$ to achieve the RMS error bound
\[
 \sqrt{\bbE_\bsDelta\big[|I_d(f) - Q_{d-1,N} (P_1 f)|^2\big]}
 \,\leq\, C N^{-1 + \delta} \quad \text{for } \delta > 0,
\]
where $C$ depends on $\delta$ and the norm $\|P_1 f\|_{\calW_{d - 1}} <
\infty$.
\end{theorem}

To our knowledge this is the first rigorous error bound giving close to
$\calO(1/N)$ convergence for a QMC rule applied to an option pricing
problem.

Practical details on how to efficiently construct a lattice generating
vector for such option pricing problems, along with a full error analysis
that is explicit in how the constant depends on the dimension (including
how to choose the weight parameters $\{\gamma_\setu\}$) will be studied in
a future paper.

\section*{Acknowledgements}

The authors acknowledge the support of the Australian Research Council
under the Discovery Project DP210100831.
The authors also thank Andreas Griewank and Hernan Le\"ovey for valuable 
discussions that encouraged this work.

\bibliographystyle{plain}

\begin{thebibliography}{10}

\bibitem{AbrSteg70}
 M.~Abramowitz and I.~A.~Stegun,
 \emph{Handbook of Mathematical Functions with Formulas, Graphs, and Mathematical Tables},
 U.S. Government Printing Office (1970).

\bibitem{ACN13a}
 N.~Achtsis, R.~Cools, and D.~Nuyens,
 \emph{Conditional sampling for barrier option pricing under the LT method},
 SIAM J.\ Financial Math.~\textbf{4} (2013), 327--352.

\bibitem{ACN13b}
 N.~Achtsis, R.~Cools, and D.~Nuyens,
 \emph{Conditional sampling for barrier option pricing under the Heston
 model}, in: J.~Dick, F.Y.~Kuo, G.W.~Peters, I.H.~Sloan (eds.), {M}onte {C}arlo
 and Quasi-{M}onte {C}arlo Methods 2012, pp.~253--269, Springer-Verlag,
 Berlin/Heidelberg (2013).

\bibitem{AcBroadGlass97}
 P.~A.~Acworth, M.~Broadie and P.~Glasserman,
 \emph{A Comparison of some Monte Carlo and quasi-Monte Carlo techniques for option pricing},
 in H.~Niederreiter, P.~Hellekalek, G.~Larcher, P.~Zinterhof (eds.),
 Monte Carlo and Quasi-Monte Carlo Methods 1996, pp.~1--18, Springer, NY (1997).

\bibitem{BST17}
 C.~Bayer, M.~Siebenmorgen, and R.~Tempone,
 \emph{Smoothing the payoff for efficient computation of Basket option
 prices}, Quantitive Finance, published online 20 July 2017, 1--15.

\bibitem{BoyBroadGlass97}
 P.~Boyle, M.~Broadie and P.~Glasserman,
 \emph{Monte Carlo methods for security pricing},
 J.~Econ.~Dyn.~Control \textbf{21} (1997), 1267--1321.

\bibitem{CMO97}
R.~E.~Caflisch, W.~Morokoff and A.~B.~Owen,
\emph{Valuation of mortgage backed securities using Brownian bridges reduces to effective dimension},
J.~Comput.~Finance \textbf{1} (1997), 27--46.

\bibitem{DKS13}
  J.~Dick, F.~Y.~Kuo, and I.~H.~Sloan,
  \emph{High dimensional integration -- the quasi-Monte Carlo way},
  Acta Numer.\ 22 (2013), 133--288.

\bibitem{DP10} J.~Dick, and F.~Pillichshammer,
\emph{Digital Nets and Sequences: Discrepancy Theory and Quasi-Monte Carlo Integration},
Cambridge University Press, NY, (2010).

\bibitem{GilWas17}
 A.~D.~Gilbert and G.~W.~Wasilkowski,
 \emph{Small superposition dimension and active set construction for multivariate
 integration under modest error demand},
 J.~Complexity \textbf{42} (2017), 94--109.

\bibitem{Glasserman} P.~Glasserman,
\emph{Monte Carlo Methods in Financial Engineering}, Springer-Verlag, Berlin/Heidelberg (2003).

\bibitem{GlaSta01}
 P.~Glasserman and J.~Staum,
 \emph{Conditioning on one-step survival for
 barrier option simulations}, Oper.\ Res.~\textbf{49} (2001), 923--937.

\bibitem{GnewHefHinRit17}
M.~Gnewuch, M.~Hefter, A.~Hinrichs and K.~Ritter,
\emph{Embeddings of weighted Hilbert spaces and applications to infinite-dimensional integration},
J.~Approx.~Theory \textbf{222} (2017), 8--39.

\bibitem{GnewHefHinRitWas17}
M.~Gnewuch, M.~Hefter, A.~Hinrichs, K.~Ritter and G.~W.~Wasilkowski,
\emph{Equivalence of weighted anchored anchored and ANOVA spaces of functions with
mixed smoothness in $L_p$},
J.~Complexity \textbf{40} (2017), p. 78--99.

\bibitem{GnewHefHinRitWas19}
M.~Gnewuch, M.~Hefter, A.~Hinrichs, K.~Ritter and G.~W.~Wasilkowski,
\emph{Embeddings for infinite-dimensional integration and $L_2$-approximation with
increasing smoothness},
J.~Complexity \textbf{54} (2019).

\bibitem{GKS10} M.~Griebel, F.~Y.~Kuo, and I.~H.~Sloan,
\emph{The smoothing effect of the ANOVA decomposition},
    J.~Complexity \textbf{26} (2010), 523--551.

\bibitem{GKS13} M.~Griebel, F.~Y.~Kuo, and I.~H.~Sloan,
\emph{The smoothing effect of integration in $\bbR^d$ and the ANOVA decomposition},
    Math.\ Comp.\ \textbf{82} (2013), 383--400.

\bibitem{GKS17note} M.~Griebel, F.~Y.~Kuo, and I.~H.~Sloan, \emph{Note on
    ``The smoothing effect of integration in~$\bbR^d$ and the ANOVA
    decomposition''}, Math.\ Comp.\ \textbf{86} (2017), 1847--1854.

\bibitem{GKS17} M.~Griebel, F.~Y.~Kuo, and I.~H.~Sloan,
\emph{The ANOVA decomposition of a non-smooth function of infinitely many variables
    can have every term smooth}, Math.\ Comp.\ \textbf{86} (2017), 1855--
    1876.

\bibitem{GKLS18}
 A.~Griewank, F.~Y.~Kuo, H.~Le\"ovey, and I.~H.~Sloan,
 \emph{High dimensional integration of kinks and jumps --- smoothing by preintegration},
 J.\ Comput.~App.~Math.\ \textbf{344} (2018), 259--274.

\bibitem{HefRit15}
M.~Hefter and K.~Ritter,
\emph{On embeddings of weighted tensor product Hilbert spaces},
J.~Complexity \textbf{31} (2015), 405--423.

\bibitem{HefRitWas16}
M.~Hefter, K.~Ritter and  G.~W.~Wasilkowski,
\emph{On equivalence of weighted anchored and ANOVA spaces of functions with mixed smoothness
of order one in $L_1$ or $L_\infty$}
J.~Complexity, \textbf{32} (2016), 1--19.

\bibitem{HMOU16}
A.~Hinrichs, L.~Markhasin, J.~Oettershagen and T.~Ullrich,
\emph{Optimal quasi-Monte Carlo rules on order 2 digital nets for the numerical integration 
of periodic functions},
Numer. Math. \textbf{134} (2016), 163--196.

\bibitem{HinSchneid16}
A.~Hinrichs and J.~Schneider,
\emph{Equivalence of anchored and ANOVA spaces via interpolation},
J.~Complexity \textbf{33} (2016), 190--198.

\bibitem{Hol11}
  M.~Holtz,
  \emph{Sparse Grid Quadrature in High Dimensions with Applications in Finance and Insurance}
  (PhD thesis), Springer-Verlag, Berlin, 2011.

\bibitem{KritzPillWas16}
 P.~Kritzer, F.~Pillichshammer and G.~W.~Wasilkowski,
 \emph{Very low truncation dimension for high dimensional integration under modest error demand},
 J.~Complexity \textbf{35} (2016), 63--85.

\bibitem{KritzPillWas17}
P.~Kritzer, F.~Pillichshammer and G.~W.~Wasilkowski,
\emph{A note on equivalence of anchored and ANOVA space; lower bounds},
J.~Complexity \textbf{38} (2017), 31--38.

\bibitem{KSWW10} F.~Y.~Kuo, I.~H.~Sloan, G.~W.~Wasilkowski and
    H.~W\'ozniakoski,
    \emph{On decompositions of multivariate functions},
    Math.\ Comp.~\textbf{79} (2010), 953--966.

\bibitem{LemLEc98}
 C.~Lemieux and P.~L'Ecuyer,
 \emph{Efficiency improvement by lattice rules for pricing Asian options},
 in D.~J.~Medeiros, E.~F.~Watson, J.~S.~Carson, M.~S.~Manivannan (eds.),
 Proceedings of the 1998 Winter Simulation Conference, 579--585,
 IEEE Computer Society Press, Washington, DC (1998).

\bibitem{Mazya11} V.~G.~Maz'ya, \emph{Sobolev Spaces.} Springer-Verlag,
    Berlin Heidelberg, Germany, 2011.

\bibitem{NK14} J.~A.~Nichols and F.~Y.~Kuo, \emph{Fast CBC construction of
    randomly shifted lattice rules achieving $O(n^{-1+\delta})$
    convergence for unbounded integrands over $\R^d$ in weighted spaces
    with POD weights},  J.\ Complexity\ \textbf{30} (2014), 444--468.

\bibitem{Nie92}
H.~Niederreiter,
\emph{Random Number Generation and Quasi-Monte Carlo Methods},
SIAM, (1992).


\bibitem{NuyWat12}
 D.~Nuyens and B.~J.~Waterhouse,
 \emph{A global adaptive quasi-Monte Carlo algorithm for functions of low truncation dimension
 applied to problems from finance}, in: L.~Plaskota and H.~Wo\'zniakowski (eds.), {M}onte {C}arlo
 and Quasi-{M}onte {C}arlo Methods 2010, pp.~589--607, Springer-Verlag,
 Berlin Heidelberg (2012).

\bibitem{PapTraub96}
 A.~Papageorgiou and J.~F.~Traub,
 \emph{Beating Monte Carlo},
 Risk \textbf{9} (1996), 53--65.

\bibitem{PasTraub95}
S.~H.~Paskov and J.~F.~Traub,
\emph{Faster valuation of financial derivatives},
J.~Portf.~Manag.~\textbf{22} (1995), 113--120.

\bibitem{SJ94}
I.~H.~Sloan and S.~Joe,
\emph{Lattice methods for Multiple Integration},
Oxford University Press, (1994).

\bibitem{WangFang03}
 X.~Wang and K-T.~Fang,
 \emph{The effective dimension and quasi-Monte Carlo integration},
 J.~Complexity \textbf{19} (2003), 101--124.

\bibitem{WangSloan06}
X.~Wang and I.~H.~Sloan,
\emph{Efficient weighted lattice rules with applications to finance},
SIAM J.~Sci.~Comp.~\textbf{28} (2006), 728--750.

\bibitem{WWH17}
 C.~Weng, X.~Wang, and Z.~He,
 \emph{Efficient computation of option prices and greeks by quasi-Monte Carlo
 method with smoothing and dimension reduction},
 SIAM J.\ Sci.\ Comput.\ \textbf{39} (2017), B298--B322.

\end{thebibliography}

\end{document}